\documentclass[10pt,a4paper]{amsart}
\usepackage[utf8]{inputenc}

\usepackage{amsmath}
\usepackage{amsfonts}
\usepackage{amssymb}
\usepackage{amsthm}
\usepackage{thmtools}
\usepackage{mathtools}
\usepackage[nameinlink]{cleveref}
\usepackage{enumitem} 
\usepackage{epigraph}
\usepackage[color=blue!20]{todonotes}
\setuptodonotes{inline}
\usepackage{stmaryrd}
\usepackage{tikz-cd}
\usepackage[nocompress]{cite}
\usepackage[new]{old-arrows}

\newtheorem{theorem}{Theorem}[]
\newtheorem{conjecture}{Conjecture}[]
\newtheorem{lemma}{Lemma}[section]

\newtheorem{proposition}[lemma]{Proposition}

\theoremstyle{definition}
\newtheorem{definition}[lemma]{Definition}
\newtheorem{remark}[lemma]{Remark}

\newcommand{\edge}{\operatorname{edge}\,}
\newcommand{\vertt}{\operatorname{vert}\,}
\newcommand{\vertti}{\operatorname{vert}_0}
\newcommand{\verttt}{\operatorname{vert}_1}
\newcommand{\Aut}{\operatorname{Aut}}
\newcommand{\st}{\operatorname{star}}
\newcommand{\core}{\operatorname{core}}
\newcommand{\Homeo}{\operatorname{Homeo}}
\newcommand{\cyc}{\operatorname{cyc}}
\newcommand{\supp}{\operatorname{supp}}

\newcommand{\mcal}{\mathcal}
\newcommand{\ve}{\varepsilon}
\renewcommand{\phi}{\varphi}

\title[On context-free subgroups]{On context-free subgroups and R. Thompson's group $V$}
\author{Henry Jaspars}
\address{Trinity College\\Cambridge\\CB2 1TQ\\United Kingdom}
\email{hj388@cam.ac.uk}

\subjclass[2020]{Primary 20F10; Secondary 20F65, 68Q45, 68Q70}
\keywords{Context-free languages, Membership problem, Transition groups, R. Thompson's group, Combinatorics on words}

\begin{document}
\maketitle

\epigraph{A\textsc{lgebra} is a general method of computation... an U\textsc{niversal} A\textsc{rithmetick}}{C\textsc{olin} M\textsc{aclaurin},\\ A Treatise on Algebra}

\begin{abstract}
Consider $V_*$, the subgroup of R. Thompson's group $V$ which stabilises $0^\omega$ under the natural action upon the Cantor set, $\{0, 1\}^\omega$. Let $G_*$ be any subgroup of a finitely generated group $G$. We show that $G_*$ is a context-free subgroup of $G$ if and only if $G_*$ is a pullback of $V_*$ under a homomorphism $G \rightarrow V$. In particular, this shows the existence of a hardest context-free membership problem. As a consequence, we prove that a group $G$ embeds into $V$ if and only if it is the transition group of a finite union of context-free automata, or equivalently, if there exist finitely many context-free subgroups of $G$ whose cores intersect trivially.
\end{abstract}

\section{Introduction}

Throughout, let $V$ refer to R. Thompson's group and $V_*$ to the stabiliser of $0^\omega$ under the natural action of $V$ upon the Cantor set $\{0, 1\}^\omega$ \cite{cfp}. We prove the following.

\begin{theorem}\label{thm:main}
Let $G$ be a finitely generated group and let $G_*$ be a subgroup of $G$. Then $G_*$ is a context-free subgroup of $G$ if and only if there exists a homomorphism $\phi: G \rightarrow V$ with $\phi^{-1}\,V_* = G_*$. 
\end{theorem}
This characterises exactly the class of context-free subgroups, originally studied by Ceccherini-Silberstein and Woess in \cite{silberstein}, and gives a relative counterpart to the classical theorem of Muller and Schupp in \cite{ms1, ms2, ms3}. Moreover, this proves the existence of a hardest context-free membership problem in the sense of Greibach \cite{hardest}.

\Cref{thm:main} generalises the following relationship between context-free languages and the dynamics of $V$ first discussed by Bennett and Bleak in \cite[Theorem 1]{demonstrative}. We call an embedding $\varphi: G \hookrightarrow V$ \emph{demonstrative} if there exists an open set $U \subset \{0, 1\}^\omega$ such that $U \cap U \cdot \varphi(g) = \emptyset$ whenever $g \neq 1$. We then have the following.

\begin{theorem}
Let $G$ be a finitely generated group. Then $G$ has context-free word problem if and only if there exists a demonstrative embedding $\varphi: G \rightarrow V$. 
\end{theorem}

We also have the following corollary of \Cref{thm:main}.

\begin{theorem}\label{cor:main}
Let $G$ be a finitely generated group. Then there exist finitely many context-free subgroups of $G$ whose cores intersect trivially if and only if $G$ embeds into $V$.
\end{theorem}

This characterises exactly the class of transition groups of finite unions of context-free automata first described in \cite{cftr}, as well as providing a structure theorem for finitely generated subgroups of $V$. We remark that $V$ is not unique in this regard, as there are several groups which are known to be bi-embeddable with $V$; see for instance \cite{bleakcocf}.

In both \Cref{thm:main} and \Cref{cor:main}, results in the spirit of the `if' directions are already known due to \cite{lehnert, birget}. We nonetheless include their proofs in \Cref{section:pda} for the sake of completeness.

We are motivated in large part by the following conjecture of Lehnert in \cite{lehnert}, rephrased as follows by Bleak in \cite{bleakcocf}. Here, the \emph{co-word problem} refers to the complement of the word problem. 

\begin{conjecture}[Lehnert's Conjecture]
Let $G$ be a finitely generated group. Then $G$ has context-free co-word problem if and only if $G$ embeds into $V$. 
\end{conjecture}

By \Cref{cor:main}, we may rephrase Lehnert's conjecture in the following, purely language-theoretic setting.

\begin{conjecture}
Let $G$ be a finitely generated group. Then $G$ has context-free co-word problem $\mcal{L} \subset \Sigma^*$ if and only if there exist finitely many co-membership problems $\mcal{L}_i$ of subgroups of $G$ such that 
$$\bigcup_{\sigma \in \Sigma^*} \bigcup_{i = 1}^n \overline{\sigma} \, \mcal{L}_i^{\cyc} \sigma = \mcal{L}$$
where $\mcal{L}_i^{\cyc} \subset \Sigma^*$ denotes the set of cyclic shifts of words in $\mcal{L}_i$. 
\end{conjecture}
The `if' direction, namely, that transition groups of finite unions of context-free graphs have context-free co-word problem, is known due to \cite[Theorem 4.11]{cftr}.

Finally, \Cref{thm:main} allows us to construct interesting examples of context-free subgroups arising from the dynamics of $V$ on $\{0, 1\}^\omega$.

\begin{theorem}\label{thm:wreath}
Let $G$ and $H$ be finitely generated groups and let $G_*$ and $H_*$ be context-free subgroups of $G$ and $H$ respectively. Let
$$G \wr_{H_*} H = \bigg(\bigoplus_{H_* \setminus H} G\bigg) \rtimes H$$
denote the generalised wreath product. Then 
$$\bigg(\bigg(\bigoplus_{H_* \setminus (H - H_*)} G \bigg)\oplus G_*\bigg) \rtimes H_*$$
is a context-free subgroup of $G \wr_{H_*} H$.
\end{theorem}

After the proof was complete, the author learned that the same results have been proved independently by Bodart, D’Angeli, Perego and Rodaro \cite{independent}. 

\subsection*{Acknowledgements}
The author wishes to thank Henry Wilton for his immense generosity and wisdom throughout the course of this project. The author also wishes to thank Collin Bleak, who first introduced the author to this constantly beautiful, compelling and surprising area of study and Francesco Fournier-Facio, Alex Levine and Aidan Evans for many useful conversations and comments on earlier iterations of this manuscript. The author would finally like to thank Trinity College's Summer Studentship programme for providing generous funding for this project.

\section{Preliminaries}

Throughout, we largely follow the conventions of \cite{wpig} and \cite{trees}.

\begin{definition}[Languages and the Membership Problem] An \emph{alphabet} (for instance, $\Sigma$) is a finite collection of \emph{letters} (for instance, $s$). \emph{Finite words} are finite sequences $\sigma$ of elements of $\Sigma$. The set of all finite words over $\Sigma$ is denoted $\Sigma^*$, which is a monoid under concatenation, with identity $\ve$, the \emph{empty word}. We will abuse notation and identify $\Sigma$ with the subset of $\Sigma^*$ of words of length $1$. For $\sigma', \sigma'' \in \Sigma^*$, if $\sigma = \sigma' \, \sigma''$, we say that $\sigma'$ is a \emph{prefix} of $\sigma$, while $\sigma''$ is a \emph{suffix}. A \emph{language} over $\Sigma$ is a subset of $\Sigma^*$.

We similarly define the set of \emph{right-infinite} words over $\Sigma$, denoted $\Sigma^\omega$, to be the set of all right-infinite sequences of letters in $\Sigma$. The union $\Sigma^* \cup \Sigma^\omega$ is denoted $\Sigma^{\leq \omega}$. For $\sigma \in \Sigma^*$ distinct from the empty word, we denote by $\sigma^\omega$ the right-infinite word which is the concatenation of countably many copies of $\sigma$. There is a natural monoidal left-action of $\Sigma^*$ upon $\Sigma^\omega$, again by concatenation. The \emph{Cantor topology} upon $\Sigma^\omega$ is the topology with basis comprising \emph{prefix cones}, $\sigma \,\Sigma^\omega$ for $\sigma \in \Sigma^*$. 

We say that an alphabet $\Sigma$ is \emph{symmetric} if it is endowed with an involution 
$$\Sigma \longrightarrow \Sigma, \qquad s \longmapsto \overline{s}$$
which extends to an involution 
$$\Sigma^* \longrightarrow \Sigma^*, \qquad s_1 \hdots s_n \longmapsto \overline{s_1 \hdots s_n} = \overline{s_n} \hdots \overline{s_1}$$
for $s_i \in \Sigma$. 

Let $G$ be a finitely generated group with $G_*$ a subgroup of $G$, and let $\Sigma$ be a symmetric alphabet. We call a homomorphism $\rho: \Sigma^* \rightarrow G$ \emph{symmetric} if for every $\sigma \in \Sigma^*$, $\rho(\overline{\sigma}) = \rho(\sigma)^{-1}$. For a symmetric, surjective homomorphism $\rho: \Sigma^* \rightarrow G$, the \emph{membership problem} of $G_*$ in $G$ with respect to $\rho$ is the language $\rho^{-1} \, G_* \subset \Sigma^*$. In the case that $G_*$ is trivial, this is usually referred to as the \emph{word problem} of $G$ with respect to $\rho$.
\end{definition}

We make the following remark. 

\begin{remark}\label{prop:invhom1}
Let $G$ and $H$ be finitely generated groups, and let $\Sigma_G$ and $\Sigma_H$ be finite symmetric alphabets. For a homomorphism $\phi: G \rightarrow H$ and symmetric, surjective monoid homomorphisms $\rho_G: \Sigma_G^* \rightarrow G$ and $\rho_H: \Sigma_H^* \rightarrow H$, there exists a symmetric monoid homomorphism $\psi: \Sigma_G^* \rightarrow \Sigma_H^*$ such that the diagram

\begin{center}
\begin{tikzcd}
	{\Sigma_G^*} & G \\
	{\Sigma_H^*} & H
	\arrow["{\rho_G}", from=1-1, to=1-2]
	\arrow["{\psi}"', from=1-1, to=2-1]
	\arrow["\phi", from=1-2, to=2-2]
	\arrow["{\rho_H}", from=2-1, to=2-2]
\end{tikzcd}
\end{center}
commutes. Moreover, if $H_*$ is a subgroup of $H$ and $G_* = \phi^{-1} \, H_*$, then the membership problem $\rho_G^{-1} G_*$ is the preimage of the membership problem of $\rho_H^{-1} \, H_*$ under $\psi$. It is in this sense that the membership problem for $H_*$ in $H$ is \emph{harder} than that of $G_*$ in $G$ in the sense of Greibach \cite{hardest}, since in order to decide whether a word $\sigma \in \Sigma_H^*$ has $\sigma \in \rho_H^{-1} \, H_*$, it suffices to decide whether $\psi(\sigma) \in \rho_G^{-1} \, G_*$. 
\end{remark}

\begin{definition}[R. Thompson's group $V$]\label{def:thompsv}
A \emph{prefix partition} is a finite language $\mcal{A} \subset \{0, 1\}^*$ such that the cones of the form $\alpha \, \{0, 1\}^\omega$ for $\alpha \in \mcal{A}$ partition $\{0, 1\}^\omega$. The \emph{prefix partition map} induced by a bijection $f:\mcal{A} \rightarrow \mcal{B}$ for $\mcal{A}, \mcal{B} \subset \{0, 1\}^*$ prefix partitions, is the map
$$[f]: \{0, 1\}^\omega \longrightarrow \{0, 1\}^\omega, \qquad \alpha \, \eta \longmapsto f(\alpha) \, \eta$$
for $\alpha \in \mcal{A}$ and $\eta \in \{0, 1\}^\omega$. \emph{R. Thompson's group $V$} is the group of prefix partition maps and acts upon $\{0, 1\}^\omega$ on the right.
\end{definition}

\begin{remark}\label{rem:vprops}
Throughout, we assume several well-known properties of $V$: firstly, that $V$ is a finitely generated group; secondly, that the orbit of $0^\omega$ is the set of words of the form $\alpha \, 0^\omega$ for $\alpha \in \{0, 1\}^*$, which is a dense subset of $\{0, 1\}^\omega$; and finally that every finite direct sum $V \oplus V \oplus \hdots \oplus V$ embeds into $V$. We refer the reader to \cite{cfp} for proofs of these facts.
\end{remark}

\begin{definition}[Automata] An \emph{automaton} is a directed graph $\Gamma$ comprising \emph{states} $X = \vertt \Gamma$ and \emph{transitions} $Y = \edge \Gamma$, in addition to a map 
$$Y \longrightarrow X \times X, \qquad e \longmapsto (\iota(e), \tau(e))$$
sending a unique transition to its \emph{initial} and \emph{terminal states} respectively. We often write that $e$ is a transition \emph{from} $\iota(e)$ \emph{to} $\tau(e)$. In addition, we may wish to have a \emph{labelling} 
$$Y \longrightarrow \Sigma \cup \{\ve\}, \qquad e \longmapsto \ell(e)$$
as well as two sets of \emph{initial} and \emph{terminal states} $\vertti \Gamma, \verttt \Gamma \subset \vertt \Gamma$ respectively, where we usually insist (unless otherwise stated) that $\vertti \Gamma$ is a singleton. An automaton without a labelling is called \emph{unlabelled}, while an automaton whose labels belong entirely to $\Sigma$ is called \emph{proper}. 
We call a proper automaton \emph{symmetric} if $\Sigma$ is a symmetric alphabet and there exists a fixed-point free involution
$$Y \longrightarrow Y, \qquad e \longmapsto \overline{e}$$
such that $\iota(\overline{e}) = \tau(e)$ and $\ell(\overline{e}) = \overline{\ell(e)}$.

A \emph{morphism} of automata $f: \Gamma \rightarrow \Delta$ comprises a pair of functions from transitions to transitions and states to states, which preserves the full structure of the underlying graph, the labelling (where it exists) and sends initial states to initial states and terminal states to terminal states. We define \emph{isomorphisms} of automata in the usual way. 
\end{definition}

\begin{definition}[Stars, Determinism and Transition groups]
Let $v$ a state of an automaton $\Gamma$. The \emph{star} of $v$ in $\Gamma$ is the set of transitions 
$$\st(v, \Gamma) = \{e \in \edge \Gamma: \iota(e) = v\},$$
while for $s \in \Sigma \cup \{\ve\}$, the \emph{$s$-star} of $v$ in $\Gamma$ is the set of transitions
$$\st_s(v, \Gamma) = \{e \in \edge \Gamma: \iota(e) = v,\, \ell(e) = s\}.$$
We call a proper automaton \emph{locally finite} if $\st(v, \Gamma)$ is finite for every state $v$. A map of automata $f: \Gamma \rightarrow \Delta$ yields, for every state $v$ of $\Gamma$, a map
$$f_v: \st(v, \Gamma) \longrightarrow \st(f(v), \Delta).$$
We call $f$ a \emph{covering} if $f_v$ is always a bijection and an \emph{immersion} if $f_v$ is always an injection.

The labelling of a proper automaton $\Gamma$ may be understood as a morphism of automata $\ell: \Gamma \rightarrow |\Sigma|$, where $|\Sigma|$ the unique automaton with 
$$\vertt |\Sigma| = \{*\}, \qquad \edge |\Sigma| = \Sigma.$$
We call a proper automaton \emph{deterministic} if $\ell$ is an immersion and \emph{fully deterministic} if $\ell$ is a covering.

Let $X$ be a set. There is a natural correspondence between right-actions of $\Sigma^*$ upon $X$ and fully deterministic automata $\Gamma$ with $X = \vertt \Gamma$ (up to choice of initial and terminal states), where $\Gamma$ has a transition from $x$ to $x \cdot s$ labelled $s$ for all $x \in X$, $s \in \Sigma$. Moreover, if $\Gamma$ is symmetric, this descends to a group action. The induced subgroup of $\Aut(X)$ from this action is called the \emph{transition group} of $\Gamma$.
\end{definition}

\begin{definition}[Paths and Languages]
Let $\Gamma$ be an automaton. A \emph{finite path} in $\Gamma$ of \emph{length $n$} is a finite sequence $\pi = e_1  \hdots e_n$ of transitions of $\Gamma$ such that $\tau(e_i) = \iota(e_{i + 1})$ for all indices $i = 1, \hdots, n - 1$. If further $\tau(e_n) = \iota(e_1)$, we say that $\pi$ is a \emph{loop}. We define $\iota(\pi) = \iota(e_1)$, $\tau(\pi) = \tau(e_n)$ and 
$$\ell(\pi) = \ell(e_1) \hdots \ell(e_n),$$
and let $\pi\, \pi'$ denote the concatenation of paths $\pi, \pi'$ whenever $\tau(\pi) = \iota(\pi')$. We similarly identify the edges of $\Gamma$ with the set of paths in $\Gamma$ of length $1$, and define \emph{prefixes} and \emph{suffixes} in the same fashion as we did for words. The set of all finite paths $\pi$ in $\Gamma$ is denoted by $\Gamma^*$, while the subset of paths with $\iota(\pi) \in \vertti \Gamma$ is denoted $\Gamma_0^*$. We similarly define \emph{right-infinite paths} in $\Gamma$ as right-infinite sequences $\pi = e_1  e_2 \hdots$ of transitions of $\Gamma$ such that $\tau(e_i) = \iota(e_{i + 1})$ for all $i \geq 1$, and denote the set of all right-infinite paths in $\Gamma$ by $\Gamma^\omega$. We define $\Gamma_0^\omega$, $\Gamma^{\leq \omega}$ and $\Gamma^{\leq \omega}_0$ in a similar fashion to above. The \emph{Cantor topology} upon $\Gamma^\omega_0$ is the topology whose basis comprises \emph{prefix cones}, $\pi \, \Gamma^\omega$, for $\pi \in \Gamma^*_0$.

We call an automaton \emph{accessible} if for every state $v$ of $\Gamma$ there exists a path from a state of $\vertti \Gamma$ to $v$. Henceforth all automata will be considered accessible unless otherwise stated. 

For an automaton $\Gamma$, the (existential) language \emph{accepted} by $\Gamma$ is defined to be
$$\mcal{L}(\Gamma) = \Bigg\{\,
\sigma \in \Sigma^*: 
\begin{aligned}
& \quad \quad \, \textup{there exists a path } \pi \textup{ such that } \\
& \iota(\pi) \in \vertti \Gamma, \ell(\pi) = \sigma \textup{ and } \tau(\pi) \in \verttt \Gamma
\end{aligned}
\Bigg\} \subset \Sigma^*.$$
\end{definition}

A key example is the following. 

\begin{remark}[Schreier automata]\label{rem:schreier}
Let $G$ be a finitely generated group, and let $\rho: \Sigma^* \rightarrow G$ be a surjective monoid homomorphism. Let $G_*$ be a subgroup of $G$. Define the \emph{Schreier automaton} of $G_*$ in $G$ to be the automaton $\Gamma$ whose states are the right cosets $G_* \setminus G$ and whose transitions are taken from $G_* g$ to $G_* (g \, \rho(s))$ with label $s$ for all $g \in G$ and $s \in \Sigma$. Further let $\vertti \Gamma = \verttt \Gamma = \{G_*\}$. The automaton $\Gamma$ is proper, symmetric, accessible and fully deterministic, and $\mcal{L}(\Gamma)$ coincides with the membership problem for $G_*$ in $G$. Moreover, the action of $G$ upon the states of $\Gamma$ induces a surjective homomorphism from $G$ to the transition group of $\Gamma$ with kernel
$$\core(G_*) = \bigcap_{g \in G}g^{-1} \, G_* \, g,$$ 
the \emph{core} of $G_*$ in $G$. Observe that every proper, symmetric, accessible, fully deterministic automaton $\Gamma$ with $\vertti \Gamma = \verttt \Gamma$ is a Schreier automaton.
\end{remark}

We now make the central definition of the paper.

\begin{definition}[Context-free automata]\label{def:cfa}
Let $\Gamma$ be a locally finite, accessible automaton and let $v$ a state of $\Gamma$. Define a metric $d$ upon $\vertt \Gamma$, where $d(v, w)$ is the length of the shortest path from $v$ to $w$, where here we allow paths to be undirected. Let $|v|$ be the length of the shortest path from a state of $\vertti \Gamma$ to $v$. Define $\nabla^n \Gamma$ to be the automaton induced upon the states
$$\vertt \nabla^n \Gamma = \{v \in \vertt \Gamma: |v| \geq n\}$$
with 
$$\vertti \nabla^n \Gamma = \{v \in \vertt \Gamma: |v| = n\}$$
(called \emph{frontier points} in \cite{ms3}) and 
$$\verttt \nabla^n \Gamma = \{v \in \verttt \Gamma: |v| \geq n\}.$$
Henceforth, we write $\nabla = \nabla^1$. 

For a state $v$ with $|v| = n$, the (undirected) connected component of $\nabla^n \Gamma$ containing $v$ is called the \emph{context} of $v$ in $\Gamma$, denoted $\partial^v \Gamma$ (these are referred to as \emph{ends} in \cite{ms2, ms3}, a term we do not use for avoidance of confusion). We call $\Gamma$ \emph{context-free} if there are finitely many contexts up to isomorphism. Languages of the form $\mcal{L}(\Gamma)$ for $\Gamma$ context-free are called \emph{context-free languages}. (We detail in \Cref{section:pda} how this coincides with the usual definition from computer science).
\end{definition}

We have the following classical proposition due to \cite[2.6.34]{cocf}.

\begin{proposition}\label{prop:invhom2}
Let $\mcal{L} \subset \Sigma_H^*$ be a context-free language and let $\psi: \Sigma_G^* \rightarrow \Sigma_H^*$ be a monoid homomorphism. Then $\psi^{-1}\,\mcal{L} \subset \Sigma_G^*$ is a context-free language. 
\end{proposition}

We then have the following results upon context-free membership problems.

\begin{proposition}\label{prop:pullback}
Let $G, H$ be finitely generated groups with subgroups $G_*, H_*$ and let $\rho_G, \, \rho_H,\,  \phi$ and $\psi$ be as in \Cref{prop:invhom1}. Suppose the membership problem of $H_*$ in $H$ is context-free. Then the membership problem of $G_*$ in $G$ is context-free. In particular, the membership problem of $G_*$ in $G$ being a context-free language is independent of the choice of symmetric alphabet $\Sigma$ and surjective monoid homomorphism $\rho: \Sigma^* \rightarrow G$.
\end{proposition}

\begin{proof}
Immediate from \Cref{prop:invhom1} and \Cref{prop:invhom2}.
\end{proof}

In particular, we may suppress any mention of $\rho: \Sigma^* \rightarrow G$ and simply call $G_*$ a \emph{context-free subgroup} of $G$. 

We have the following result of \cite[Corollary 4.10]{silberstein}.

\begin{proposition}\label{lem:Schreier}
Let $G$ be a finitely generated group and $G_*$ a subgroup of $G$. Let $\Sigma$ be a symmetric alphabet, and let $\rho: \Sigma^* \rightarrow G$ be a symmetric, surjective homomorphism. Then $G_*$ is a context-free subgroup of $G$ if and only if the Schreier automaton $\Gamma$ of $G_*$ in $G$ is context-free with respect to $\rho$.
\end{proposition}

In particular, this shows that $G_*$ has context-free membership problem in $G$ if and only if it has \emph{deterministic} context-free membership problem in $G$. Moreover, per \Cref{rem:schreier}, $G$ is the transition group of a finite union of context-free automata if and only if there exist finitely many context-free subgroups $G_1, \hdots, G_n$ of $G$ such that 
$$\bigcap_{i = 1}^n \core(G_i) = 1,$$
that is, if the intersection of the cores of the groups $G_1, \hdots, G_n$ is trivial.

We also have the following result due to \cite[Proposition 4.17]{cftr}.

\begin{proposition}\label{lem:cftrsub}
Suppose $G$ is the transition group of a finite union of context-free automata and suppose $H$ is a subgroup of $G$. Then $H$ is the transition group of a finite union of context-free automata. 
\end{proposition}

\section{Pseudogroups and topological full groups}

\begin{definition}[Pseudogroups]
A \emph{pseudogroup} (following \cite[Def 3.1.1]{thurston}) upon a topological space $X$ is a set $\mcal{G}$ of homeomorphisms between open sets of $X$ satisfying the following:

\begin{enumerate}[label=(\roman*)]
\item the domains of the elements $g \in \mcal{G}$ cover $X$;
\item the restriction of an element $g \in \mcal{G}$ to any open set contained in its domain is also in $\mcal{G}$;
\item the composition $g \cdot g'$ of two elements in $\mcal{G}$, where defined, is in $\mcal{G}$;
\item the inverse of an element of $\mcal{G}$ is in $\mcal{G}$;
\item if $g: U \rightarrow V$ is a homeomorphism between open sets of $X$ and $U$ is covered by open sets $U_\alpha$ such that each restriction $g|_{U_\alpha}$ is in $\mcal{G}$, then $g \in \mcal{G}$.
\end{enumerate}
Observe that $\mcal{G}$ contains the identity map upon every open set of $X$. Throughout we insist that pseudogroups act upon the right. For a collection $\mcal{X}$ of homeomorphisms between open sets of $X$, the smallest pseudogroup containing $\mcal{X}$ is called the pseudogroup \emph{generated} by $\mcal{X}$. 
\end{definition}

Our main motivating example will be the following. For the remainder of this section, $\Delta$ will denote a finite unlabelled automaton.

\begin{definition}[R. Thompson's pseudogroups]
A \emph{prefix map} on $\Delta^\omega_0$ is a homeomorphism of the form
$$\pi' \, \Delta^\omega \longrightarrow \pi'' \, \Delta^\omega, \qquad \pi' \, \xi \longmapsto \pi'' \, \xi$$
for $\pi', \pi'' \in \Delta_0^*$ and $\xi \in \Delta^\omega$ with $\tau(\pi') = \tau(\pi'') = \iota(\xi)$. \emph{R. Thompson's $\Delta$ pseudogroup}, denoted $\mcal{G}(\Delta)$, is the pseudogroup generated by prefix maps.
\end{definition}

\begin{definition}[Topological full groups]\label{def:thompsongroup}
The \emph{topological full group} of $\mcal{G}$, denoted $[\mcal{G}]$, is the subgroup of $\Homeo(X)$ comprising homeomorphisms $g: X \rightarrow X$ such that, for all $x \in X$, there exists an open neighbourhood $U$ of $x$ such that $g|_U \in \mcal{G}$. It can be seen that this is indeed a group. \emph{R. Thompson's $\Delta$ group} is the topological full group $[ \mcal{G}(\Delta)]$. 
\end{definition}

We justify our choice of notation in the following.

\begin{proposition}\label{prop:vsagree}
For $\Delta = |\{0, 1\}|$, $[ \mcal{G}(\Delta) ] = V$ with respect to the action upon $\Delta^\omega_0 = \{0, 1\}^\omega$. 
\end{proposition}

\begin{proof}
Clearly all elements of $V$ belong to $[ \mcal{G}(\Delta)]$ for $\Delta = |\{0, 1\}|$. To see that all elements of $[ \mcal{G}(\Delta) ]$ belong to $V$, suppose $g \in [ \mcal{G}(\Delta) ]$. Then for every $\xi \in \{0, 1\}^\omega$, there exists an open neighbourhood $U_\xi$ of $\sigma$ such that $g|_{U_\xi}$ is a restriction of a prefix map. Moreover, there exists a cone $\alpha_\xi \, \{0, 1\}^\omega \subset U_\xi$ containing $\xi$. By the compactness of $\{0, 1\}^\omega$, there exists a finite number of elements $\xi_i \in \{0, 1\}^\omega$ for $i = 1, \hdots, n$ such that the cones $\alpha_{\xi_i} \, \{0, 1\}^\omega$ cover $\{0, 1\}^\omega$. Observe that every pair of cones are either disjoint, or contained in one another. Therefore, the collection of cones $\alpha_{\xi_i} \, \{0, 1\}^\omega$ which are maximal up to inclusion partitions $\{0, 1\}^\omega$, and with respect to this partition, $g$ is a prefix partition map. Therefore these groups coincide.
\end{proof}

We now introduce a useful class of automata.

\begin{definition}[Augmented automata] The \emph{augmented automaton} $\Delta_+$ is the automaton with 
$$\vertt \Delta_+ = \vertt\Delta \sqcup \{+\},\quad \vertti \Delta_+ = \vertti \Delta, \quad \verttt \Delta_+ = \{+\}$$
and with $\edge \Delta_+$ comprising $\edge \Delta$ in addition to a transition $e_v$ from each state $v$ of $\Delta_+$ to $+$. Observe that every path in $(\Delta_+)^\omega_0$ either belongs to $\Delta_0^\omega$, or can be uniquely written in the form $\pi \, e_v \, (e_+)^\omega$ for $\pi \in \Delta^*_0$ and $\tau(\pi) = v$ (that is, every path either never reaches $+$, or does so and remains there indefinitely).
\end{definition}

\begin{proposition}\label{prop:density}
The map 
$$\Delta_0^* \longhookrightarrow (\Delta_+)_0^\omega, \qquad \pi \longmapsto \pi \, e_v \, (e_+)^\omega$$
for all $\pi \in \Delta_0^*$ where $v = \tau(\pi)$, has dense image all of whose points are isolated and which is preserved by the action of $[ \mcal{G}(\Delta_+) ]$.
\end{proposition}

\begin{proof}
Let $U$ be an open set of $(\Delta_+)^\omega_0$. Then $U$ has an open subset of the form $\pi \, (\Delta_+)^\omega$ for some $\pi \in (\Delta_+)^*_0$. If $\tau(\pi) = +$, then $\pi = \pi' \, e_v \, (e_+)^n$ for some $n \geq 0$ and some path $\pi'$ contained within $\Delta$ with $v = \tau(\pi)$. In this case, $\pi' \, e_v \, (e_+)^\omega$ is contained within the image of $\Delta_0^\omega$ and in $U$. Otherwise, $\pi \, e_v \, (e_+)^\omega$ is contained within the image of $\Delta^*_0$ and in $U$. Therefore the image of $\Delta_0^*$ is dense.

Isolation follows from the fact that for every $\pi \in \Delta^*_0$, the cone $\pi \, e_v \, (\Delta_+)^\omega$ is just the singleton $\{\pi \, e_v \, (e_+)^\omega\}$, the only infinite path in $\Delta_+$ with prefix $\pi \, e_v$, and hence $\pi \, e_v \, (e_+)^\omega$ is isolated.

As for being fixed by $[ \mcal{G}(\Delta_+) ]$, observe that the image under a prefix map of a path with suffix $(e_+)^\omega$ will also be a path with suffix $(e_+)^\omega$. In particular, as elements of $[ \mcal{G}(\Delta_+) ]$ are locally prefix substitutions, the result follows.  
\end{proof}

Throughout the remainder of this paper, we will identify $\Delta_0^*$ with its image in $(\Delta_+)_0^\omega$ and $\Delta_0^\omega$ with its complement. 

\begin{remark}
The action of an element of $[\mcal{G}(\Delta_+)]$ upon $\Delta_0^*$ is usually called a \emph{quasiautomorphism}, and the group $[\mcal{G}(\Delta_+)]$ the \emph{quasiautomorphism group}. For instance, in \cite{lehnertthesis, bleakcocf}, the group $[ \mcal{G}(\Delta_+) ]$ is denoted $\textrm{QAut}(\mcal{T}_{2, c})$ for $\Delta = |\{0, 1\}|$.
\end{remark}

\begin{definition}[Partition labelling]
A \emph{partition labelling} on $\Delta_+$ is a map
$$f: \edge \Delta_+ \longrightarrow \{0, 1\}^*$$
such that for each state $v$ of $\Delta$, the cones $f(e) \, \{0, 1\}^\omega$ for $e \in \st(v, \Delta_+)$ partition $\{0, 1\}^\omega$. Explicitly, we may construct such an $f$ by bijecting each $\st(v, \Delta_+)$ with a set of the form
$$\{0^k \,1 : 0 \leq k < n - 1\} \cup \{0^{n - 1}\}$$
for $\st(v, \Delta_+)$ with cardinality $n \geq 2$ and $\{\ve\}$ for $\st(v, \Delta_+)$ of cardinality $1$ (observe that $\st(v, \Delta_+)$ is always non-empty, so this is always well-defined). This extends naturally to a map
$$f: (\Delta_+)^\omega_0 \longrightarrow \{0, 1\}^{\leq \omega}, \qquad e_1 \, e_2 \hdots \longmapsto f(e_1) \, f(e_2) \hdots.$$
\end{definition}

Throughout the remainder of this section, $f: (\Delta_+)^\omega_0 \longrightarrow \{0, 1\}^{\leq \omega}$ will denote a partition labelling of $\Delta_+$.

\begin{proposition}\label{prop:shortpaths}
Let $\pi$ be a path in $(\Delta_+)^\omega_0$. Then $f(\pi) \in \{0, 1\}^*$ if and only if $\pi \in \Delta^*_0$.
\end{proposition}

\begin{proof}
Suppose $\pi = e_1 \, e_2 \hdots$ has $f(\pi) \in \{0, 1\}^*$, but that $\pi \in \Delta^\omega_0$. Then for every $i$, if $v = \iota(e_i)$, both $e_v$ and $e_i$ are contained in $\st(v, \Delta_+)$, in which case by definition of $f$, we deduce $f(e_i) \neq \ve$. Therefore 
$$f(\pi) = f(e_1) \, f(e_2) \hdots$$ 
is a right-infinite word, which is absurd. The converse (namely, that if $\pi \in \Delta^*_0$ then $f(\pi) \in \{0, 1\}^*$) follows immediately from the fact that $f(e_+) = \ve$, since $\st(+, \Delta_+) = \{e_+\}$ has cardinality $1$. 
\end{proof}

\begin{proposition}\label{prop:uniquepaths}
Every element of $\{0, 1\}^\omega$ can be written uniquely in the either the form $f(\pi)$ for $\pi \in \Delta^\omega_0$, or in the form $f(\pi) \, \eta$ for some $\pi \in \Delta^*_0$ and $\eta \in \{0, 1\}^\omega$. In particular, $f: (\Delta_+)^\omega_0 \longrightarrow \{0, 1\}^{\leq \omega}$ is an injection.
\end{proposition}

\begin{proof}
Let $\xi \in \{0, 1\}^\omega$. We construct a path $\pi$ inductively as follows. Let $\xi_0 = \xi$, and let $v_0$ be the unique state in $\vertti \Delta_+$. Inductively, let $e_{i + 1}$ be the unique element of $\st(v_i, \Delta_+)$ such that $f(e_{i + 1})$ is a prefix of $\xi_i$ and let
$$\sigma_i = f(e_{i + 1}) \, \xi_{i + 1}, \qquad v_{i + 1} = \tau(e_{i + 1}).$$
In this way, we obtain that $f(\pi) = f(e_1) \, f(e_2)\, \hdots$ is a prefix of $\xi$, as required. Uniqueness follows since, for every $i$, $e_{i+1}$ is the unique element of $\st(v_i, \Delta_+)$ such that $f(e_{i+1})$ is a prefix of $\xi_i$. 
\end{proof}

\begin{remark}
The construction of \Cref{prop:uniquepaths} is that of a \emph{rational transducer} between $\Delta_+$ and $|\{0, 1\}|$; see for instance \cite[2.4.1.2]{sakar} and \cite[Example 4.9]{bleakagain}. 
\end{remark}

\begin{definition}[Demonstrativity]
Let $G$ be a group and $G_*$ a subgroup of $G$. We say that a homomorphism $\phi: G \rightarrow V$ is \emph{demonstrative} with respect to $G_*$ for an open set $U \subset \{0, 1\}^\omega$ in the sense of \cite{demonstrative} if the following are equivalent for $g_* \in G$:
\begin{enumerate}[label = (\roman*)]
\item $U \cap U \cdot \phi(g_*) \neq \emptyset$;
\item $\phi(g_*)$ fixes $U$ pointwise;
\item $g_* \in G_*$.
\end{enumerate}
\end{definition}

We now prove the main result of the section, inspired by \cite[Theorem 18]{bleakcocf}.

\begin{lemma}\label{lem:embed}
Let $\Delta$ be an unlabelled finite automaton. Then there exists an embedding $\phi: [\mcal{G}(\Delta_+)] \hookrightarrow V$. Moreover, suppose $[\mcal{G}(\Delta_+)]_*$ is the point stabiliser of $[\mcal{G}(\Delta_+)]$ with respect to the action upon $(\Delta_+)_0^\omega$, for a point in $\Delta^*_0$. Then we have $\phi^{-1}\, V_* = [\mcal{G}(\Delta_+)]_*$. In particular, $\phi$ is demonstrative for $[\mcal{G}(\Delta_+)]$ with respect to $[\mcal{G}(\Delta_+)]_*$ for some open set $U \subset \{0, 1\}^\omega$. 
\end{lemma}

\begin{proof}
For $g: (\Delta_+)^\omega_0 \rightarrow (\Delta_+)^\omega_0$ in $[ \mcal{G}(\Delta_+) ]$, define $\phi(g): \{0, 1\}^\omega \longrightarrow \{0, 1\}^\omega$ by
$$f(\pi) \cdot \phi(g) = f(\pi \cdot g), \qquad (f(\pi) \, \eta)\cdot \phi(g) = f(\pi \cdot g) \, \eta$$
for $\pi \in \Delta^\omega_0$ and $\pi \in \Delta^*_0$, $\eta \in \{0, 1\}^\omega$ respectively. The map $\phi$ is well-defined by \Cref{prop:uniquepaths}, and is clearly a group homomorphism from $[ \mcal{G}(\Delta_+) ]$ to $\Aut(\{0, 1\}^\omega)$. Moreover, $\phi$ is injective, since if $\phi(g)$ is the identity for some $g \in [ \mcal{G}(\Delta) ]$, then $f(\pi \cdot g) = f(\pi)$ for all $\pi \in \Delta^\omega_0$ and $f(\pi \cdot g) \, 0^\omega = f(\pi) \, 0^\omega$ for all $\pi \in \Delta^*_0$, whence $g$ is the identity by \Cref{prop:uniquepaths}.

Let $g \in [ \mcal{G}(\Delta_+) ]$. We now claim that $\phi(g) \in V$. Suppose $\xi \in \{0, 1\}^\omega$. Then either $\xi = f(\pi)$ for $\pi \in \Delta^\omega_0$, or $\xi = f(\pi)\, \eta$ for $\pi \in \Delta^*_0$ and $\eta \in \{0, 1\}^\omega$. Observe that there exists an open neighbourhood $U = \pi' \, (\Delta_+)^\omega$ of $\pi$ such that $g|_U$ is a prefix map, namely,
$$g|_U: \pi' \, (\Delta_+)^\omega \longrightarrow \pi'' \, (\Delta_+)^\omega.$$
Let $\widetilde{U} = f(\pi') \, \{0, 1\}^\omega$ be a neighbourhood of $\xi$. Then $\phi(g)|_{\widetilde{U}}$ is the prefix map
$$\phi(g)|_{\widetilde{U}}: f(\pi') \, \{0, 1\}^\omega \longrightarrow f(\pi'') \, \{0, 1\}^\omega.$$
Therefore $\phi(g) \in V$. 

Finally, observe that if $[ \mcal{G}(\Delta_+) ]_*$ is the point stabiliser of a point $\pi \in \Delta^*_0$, then $\phi(g)$ stabilises $f(\pi) \, 0^\omega$ if and only if $g$ stabilises $\pi$. In particular, since $f(\pi) \, 0^\omega$ is in the orbit of $0^\omega$ under $V$ by \Cref{rem:vprops}, we deduce that $\phi^{-1}\,V_* = [ \mcal{G}(\Delta_+) ]_*$. Moreover, $\phi$ is demonstrative with respect to $[\mcal{G}(\Delta_+)]$ for the open set $U = f(\pi) \, \{0, 1\}^\omega$, so the result follows. 
\end{proof}

\section{Canonisation}\label{sec:canon}

Henceforth, let $\Gamma$ be a context-free automaton. As prior, connected components will be considered in the undirected sense, even for directed automata. 

We begin with the following observations. 

\begin{proposition}\label{prop:partial}
Let $m, n \geq 0$. Then $\nabla^m (\nabla^n \Gamma) = \nabla^{m + n} \Gamma$.
\end{proposition}

\begin{proof}
This is immediate from the fact that, if $|v| = m + n$, then there exists $w$ such that $|w| = n$ and $d(v, w) = m$ (and similarly for $|v| \geq m + n$). 
\end{proof}

\begin{proposition}\label{prop:inclusions}
Let $v$ a state of $\Gamma$ with $|v| = n$. Then there exists a unique sequence of contexts 
$$\Gamma = \partial^{v_0} \Gamma \supset \partial^{v_1}\Gamma \supset \hdots \supset \partial^{v_n} \Gamma = \partial^v \Gamma$$
with $v_0 \in \vertti \Gamma$, $v_n = v$ such that for every $i = 1, \hdots, n - 1$, $\partial^{v_{i + 1}}$ is a connected component of $\nabla(\partial^{v_i} \Gamma)$. 
\end{proposition}

\begin{proof}
Observe that $\partial^v \Gamma$ is a connected component of $\nabla^n \Gamma = \nabla(\nabla^{n - 1}\Gamma)$ by \Cref{prop:partial}. In particular, $\partial^v \Gamma$ is contained in $\nabla(\partial^w \Gamma)$ for a unique connected component $\partial^w \Gamma$ of $\nabla^{n - 1}\Gamma$, that is, with $|w| = n - 1$. The result follows by induction.
\end{proof}

Henceforth, such a sequence of contexts will be called a \emph{context chain}.

\begin{definition}[Contextual automata]
For $v$ a state of $\Gamma$, let $[\partial^v \Gamma]$ denote a choice of representative of the isomorphism class of $\partial^v \Gamma$ (that is, so that if $\partial^v \Gamma$ is isomorphic to $\partial^w \Gamma$, then $[\partial^v \Gamma] = [\partial^w \Gamma]$). For every representative $[\partial^v \Gamma]$, choose a finite collection of inclusions $e_1, \hdots, e_m$
$$e_i: [\partial^{w_i} \Gamma] \longhookrightarrow [\partial^v \Gamma]$$
such that the images of $e_i$ are connected components of $\nabla [\partial^v \Gamma]$, and every connected component of $\nabla [\partial^v \Gamma]$ is the image of a unique $e_i$. 

A \emph{contextual automaton} $\Lambda$ of $\Gamma$ is defined as follows. Let the states of $\Lambda$ be the representatives $[\partial^v \Gamma]$ and let the transitions of $\Lambda$ be the inclusions
$$e_i: [\partial^{w_i} \Gamma] \longhookrightarrow [\partial^v \Gamma]$$
for each $[\partial^v \Gamma]$ as described above, considered as a transition with initial state $[\partial^v \Gamma]$ and terminal state $[\partial^{w_i} \Gamma]$. In addition, take $\vertti \Lambda = \{[\Gamma]\}$ and $\verttt \Lambda = \emptyset$. Combinatorially, the number of edges from representatives $[\partial^v \Gamma]$ to $[\partial^w \Gamma]$ in $\Lambda$ is equal to the number of components of $\nabla[\partial^v \Gamma]$ which are isomorphic to $[\partial^w \Gamma]$. 
\end{definition}

Henceforth, fix an isomorphism $\phi_{v_0}: \Gamma \rightarrow [\Gamma]$. 

\begin{proposition}\label{prop:syntactic}
Let $\Lambda$ be a contextual automaton of $\Gamma$, and let $v$ be a state of $\Gamma$. Then there exists a unique path $\pi = e_1 \hdots e_n$ in $\Lambda^*_0$ and a unique choice of isomorphisms $\phi_{v_i}: \partial^{v_i}\Gamma \rightarrow [\partial^{v_i}\Gamma]$ for $i = 1, \hdots, n$ such that 
\begin{center}
\begin{tikzcd}
	{\Gamma=\partial^{v_0}\Gamma} & {\partial^{v_1}\Gamma} & \hdots & {\partial^{v_n}\Gamma=\partial^v \Gamma} \\
	{[\Gamma]=[\partial^{v_0}\Gamma]} & {[\partial^{v_1}\Gamma]} & \hdots & {[\partial^{v_n}\Gamma]=[\partial^v \Gamma]}
	\arrow["{\phi_{v_0}}", from=1-1, to=2-1]
	\arrow[hook', from=1-2, to=1-1]
	\arrow["{\phi_{v_1}}", from=1-2, to=2-2]
	\arrow[hook', from=1-3, to=1-2]
	\arrow[hook', from=1-4, to=1-3]
	\arrow["{\phi_{v_n}}", from=1-4, to=2-4]
	\arrow["{e_1}"', hook', from=2-2, to=2-1]
	\arrow["{e_2}"', hook', from=2-3, to=2-2]
	\arrow["{e_n}"', hook', from=2-4, to=2-3]
\end{tikzcd}
\end{center}
commutes, where $v_0 \in \vertti \Gamma$ and
$$\Gamma = \partial^{v_0} \Gamma \supset \partial^{v_1}\Gamma \supset \hdots \supset \partial^{v_n} \Gamma = \partial^v \Gamma$$
is a context chain.
\end{proposition}

\begin{proof}
Suppose $v, w$ are states of $\Gamma$ such that $\partial^w \Gamma$ is a connected component of $\nabla(\partial^v \Gamma)$ and suppose $\phi_v: \partial^v \Gamma \longrightarrow [\partial^v \Gamma]$ is an isomorphism. Then, by the definition of $\Lambda$, there exists a unique inclusion $e: [\partial^w \Gamma] \longhookrightarrow [\partial^v \Gamma]$ whose image is $\phi_v(\partial^w \Gamma)$. Therefore, there exists a unique isomorphism $\phi_w: \partial^w \Gamma \rightarrow [\partial^w \Gamma]$ such that 
\begin{center}
\begin{tikzcd}
	{\partial^{v}\Gamma} & {\partial^w\Gamma} \\
	{[\partial^{v}\Gamma]} & {[\partial^w \Gamma]}
	\arrow["{\phi_v}"', from=1-1, to=2-1]
	\arrow[hook', from=1-2, to=1-1]
	\arrow["{\phi_w}", from=1-2, to=2-2]
	\arrow["e"', hook', from=2-2, to=2-1]
\end{tikzcd}
\end{center}
commutes. The existence result then follows by \Cref{prop:inclusions} and by induction and uniqueness follows from the uniqueness of each choice of $\phi_{v_i}$, and the uniqueness in \Cref{prop:inclusions}.
\end{proof}

\begin{definition}[Charts and transition maps]
The map $\phi_v: \partial^v \Gamma \longrightarrow [\partial^v \Gamma]$ constructed in \Cref{prop:syntactic} is called the \emph{chart} upon $\partial^v \Gamma$. For isomorphic contexts $\partial^v \Gamma$ and $\partial^w \Gamma$, we call the composition
$$\phi_w^{-1} \circ \phi_v: \partial^v \Gamma \longrightarrow \partial^w \Gamma$$
the \emph{transition map} from $\partial^v \Gamma$ to $\partial^w \Gamma$. 

Henceforth, for $t = [\partial^v \Gamma]$ a context of $\Gamma$, define finite sets $Q^t_1 \subset Q^t$ by 
$$Q^t = \vertti [\partial^v \Gamma], \qquad Q^t_1 = Q^t \cap \verttt [\partial^v \Gamma].$$
For each state $v$ of $\Gamma$, identify $v$ with the pair $(\pi, q)$, which comprises the path $\pi = e_1 \hdots e_n$ as described in \Cref{prop:syntactic} and the element $q = \phi_v(v) \in Q^t$. Observe that this uniquely characterises $v$. Moreover, every pair $(\pi, q)$ with $t = \tau(\pi)$ and $q \in Q^t$ uniquely characterises a state $v$. Finally, a state $v \in \verttt \Gamma$ if and only if $v \in \verttt \partial^v \Gamma$, that is, if and only if $q = \phi_v(v) \in \verttt [\partial^v \Gamma]$.
\end{definition}

\begin{proposition}\label{prop:sub}
For every state $v = (\pi', q')$ of $\Gamma$, the states of $\partial^v \Gamma$ are precisely the states of the form $(\pi' \, \xi, q)$. Moreover, for $w = (\pi'', q'')$ with $\partial^v \Gamma$ and $\partial^w \Gamma$ isomorphic, the transition map from $\partial^v \Gamma$ to $\partial^w \Gamma$ is defined explicitly on states by
$$\phi_w^{-1} \circ \phi_v: \partial^v \Gamma \longrightarrow \partial^w \Gamma, \qquad (\pi' \, \xi, q) \longmapsto (\pi'' \, \xi, q).$$
\end{proposition}

\begin{proof}
Let $v'$ be a state of $\partial^v \Gamma$. Then $\partial^{v'} \Gamma \subset \partial^v \Gamma$. By \Cref{prop:syntactic}, there exists a diagram of the form 
\begin{equation}\label{eq:diag1}
\begin{tikzcd}
	{\Gamma=\partial^{v_0}\Gamma} & {\partial^{v_1}\Gamma} & \hdots & {\partial^{v_n}\Gamma=\partial^v \Gamma} \\
	{[\Gamma]=[\partial^{v_0}\Gamma]} & {[\partial^{v_1}\Gamma]} & \hdots & {[\partial^{v_n}\Gamma]=[\partial^v \Gamma]}
	\arrow["{\phi_{v_0}}", from=1-1, to=2-1]
	\arrow[hook', from=1-2, to=1-1]
	\arrow["{\phi_{v_1}}", from=1-2, to=2-2]
	\arrow[hook', from=1-3, to=1-2]
	\arrow[hook', from=1-4, to=1-3]
	\arrow["{\phi_{v_n}}", from=1-4, to=2-4]
	\arrow["{e_1}"', hook', from=2-2, to=2-1]
	\arrow["{e_2}"', hook', from=2-3, to=2-2]
	\arrow["{e_n}"', hook', from=2-4, to=2-3]
\end{tikzcd}
\end{equation}
with $\pi' = e_1 \hdots e_n$, such that 
$$\Gamma = \partial^{v_0} \Gamma \supset \partial^{v_1}\Gamma \supset \hdots \supset \partial^{v_n} \Gamma = \partial^v \Gamma$$
is a context chain. 

Similarly, applying \Cref{prop:syntactic} to $\partial^{v} \Gamma$, we obtain inclusions of the form 
\begin{equation}\label{eq:diag2}
\begin{tikzcd}
	{\partial^v\Gamma=\partial^{v'_0}\Gamma} & {\partial^{v'_1}\Gamma} & \hdots & {\partial^{v'_m}\Gamma=\partial^{v'} \Gamma} \\
	{[\partial^v\Gamma]=[\partial^{v'_0}\Gamma]} & {[\partial^{v'_1}\Gamma]} & \hdots & {[\partial^{v'_m}\Gamma]=[\partial^{v'} \Gamma]}
	\arrow["{\phi_{v'_0}}", from=1-1, to=2-1]
	\arrow[hook', from=1-2, to=1-1]
	\arrow["{\phi_{v'_1}}", from=1-2, to=2-2]
	\arrow[hook', from=1-3, to=1-2]
	\arrow[hook', from=1-4, to=1-3]
	\arrow["{\phi_{v'_n}}", from=1-4, to=2-4]
	\arrow["{e'_1}"', hook', from=2-2, to=2-1]
	\arrow["{e'_2}"', hook', from=2-3, to=2-2]
	\arrow["{e'_m}"', hook', from=2-4, to=2-3]
\end{tikzcd}
\end{equation}
with $\xi = e'_1 \hdots e'_m$ such that $v'_0 = v$ (without loss of generality), and such that 
$$\partial^{v}\Gamma = \partial^{v'_0}\Gamma \supset \partial^{v'_1}\Gamma \supset \hdots \supset \partial^{v'_m}\Gamma = \partial^{v'}\Gamma$$
is a context chain. In particular, concatenating diagrams (\ref{eq:diag1}) and (\ref{eq:diag2}), we obtain that $v'$ has the form $(\pi'\,\xi, q)$ for some $q$.

Let $v' = (\pi' \, \xi, q)$ be a state of $\partial^v \Gamma$. Suppose $v'$ is sent to $w'$ by the transition map $\phi_w^{-1} \circ \phi_v$. Observe that the context chain
$$\partial^v \Gamma = \partial^{v'_0}\Gamma \supset \partial^{v'_1} \Gamma \supset \hdots \supset \partial^{v'_m} \Gamma = \partial^{v'} \Gamma$$
is sent to context chain
$$\partial^w \Gamma = \partial^{w'_0}\Gamma \supset \partial^{w'_1} \Gamma \supset \hdots \supset \partial^{w'_m}\Gamma = \partial^{w'}\Gamma$$
under the transition map. In particular, the diagram
\begin{center}
\begin{tikzcd}
	{\partial^{v'_0}\Gamma} & {\partial^{v'_1}\Gamma} & \hdots & {\partial^{v'_m}\Gamma} \\
	{[\partial^{v'_0}\Gamma]=[\partial^{w'_0}\Gamma]} & {[\partial^{v'_1}\Gamma]=[\partial^{w'_1}\Gamma]} & \hdots & {[\partial^{v'_m}\Gamma]=[\partial^{w'_m}\Gamma]} \\
	{\partial^{w'_0}\Gamma} & {\partial^{w'_1}\Gamma} & \hdots & {\partial^{w'_m}\Gamma}
	\arrow["{\phi_{v'_0}}"', from=1-1, to=2-1]
	\arrow[hook', from=1-2, to=1-1]
	\arrow["{\phi_{v'_1}}"', from=1-2, to=2-2]
	\arrow[hook', from=1-3, to=1-2]
	\arrow[hook', from=1-4, to=1-3]
	\arrow["{\phi_{v'_m}}"', from=1-4, to=2-4]
	\arrow["{e'_1}"', hook', from=2-2, to=2-1]
	\arrow["{e'_2}"', hook', from=2-3, to=2-2]
	\arrow["{e'_m}"', hook', from=2-4, to=2-3]
	\arrow["{\phi_{w'_0}}", from=3-1, to=2-1]
	\arrow["{\phi_{w'_1}}", from=3-2, to=2-2]
	\arrow[hook', from=3-2, to=3-1]
	\arrow[hook', from=3-3, to=3-2]
	\arrow["{\phi_{w'_m}}", from=3-4, to=2-4]
	\arrow[hook', from=3-4, to=3-3]
\end{tikzcd}
\end{center}
commutes. Thus, in a similar fashion to above, $w'$ is identified with a state of the form $(\pi'' \, \xi, \widetilde{q})$ for some $\widetilde{q} \in Q^t$. Moreover, both $v'$ and $w'$ are sent to the same element of $\vertti [\partial^{v'} \Gamma] = \vertti [\partial^{w'} \Gamma]$ by $\phi_{v'}$ and $\phi_{w'}$ respectively. Therefore $\widetilde{q} = q$ and the result follows.
\end{proof}

We now arrive at the main result of this section, adapted from \cite[Lemma 2.3]{ms3}, which allows us to choose canonical representatives of context-free automata.

\begin{lemma}\label{lem:confreemain}
For every transition $e \in \edge \Lambda$ with $t = \tau(e)$, every $q \in Q^t$ and $s \in \Sigma$, there exists a finite set
$$\delta_s(e, q) = \{(\xi_1, q_1), \hdots, (\xi_n, q_n)\}$$
for paths $\xi_1, \hdots, \xi_n$ in $\Lambda^*$ of lengths at most $2$, with $\iota(\xi_i) = \iota(e)$ and, for states $t_i = \tau(\xi_i)$, elements $q_i \in Q^{t_i}$ such that for every path $\pi$ in $\Lambda^*_0$ with $\tau(\pi) = \iota(e)$, 
$$\tau(\st_s((\pi\,e, q), \Gamma)) = \{(\pi \,\xi_1, q_1), \hdots, (\pi \, \xi_n, q_n)\}.$$
\end{lemma}

\begin{proof}
Let $v' = (\pi \, e, q)$ be a state of $\Gamma$ and let $\partial^v \Gamma$ be the context of $\Gamma$ corresponding to $\pi$. Suppose $|v| = n$. We claim that every state $v''$ of $\tau(\st((\pi\, e, q), \Gamma))$ is identified with a pair $(\pi\, \xi, q')$ with $\xi$ a path of length at most 2. 

To see this, observe that, by the triangle inequality, $d(v, v'') \leq 2$. If $d(v, v'') = 0$, the claim is immediate. If $d(v, v'') = 1$, then $|v''| = n + 1$ and the path from $v$ to $v'$ and then to $v''$ lies within $\nabla^n \Gamma$ and so $v$ and $v''$ lie in the same connected component of $\nabla^n \Gamma$. Thus $v''$ lies in $\partial^v \Gamma$. Since $|v''| = n + 1$, $v''$ is identified with a pair $(\pi\, \xi, q')$ with $\xi$ of length $1$. We proceed similarly when $d(v, v'') = 2$.

To see that the $\xi_i$ and $q_i$ do not depend upon the choice of $\pi$, observe that for two paths $\pi'$ and $\pi''$ with $\tau(\pi') = \tau(\pi'') = \iota(e)$ and contexts $\partial^v \Gamma$ and $\partial^w \Gamma$ corresponding to $\pi'$ and $\pi''$ respectively, then the transition map $\phi_w^{-1} \circ \phi_v$ sends 
$$\st_s((\pi' \, e, q), \Gamma) \longrightarrow \st_s((\pi'' \, e, q), \Gamma).$$
The result then follows by \Cref{prop:sub}.
\end{proof} 

\section{Pumping automata}

Henceforth, let $G$ be a group with context-free subgroup $G_*$ and let $\Gamma$ be the Schreier automaton of $G_*$ in $G$ with respect to some symmetric, surjective monoid homomorphism $\rho: \Sigma^* \rightarrow G$. Recall that by \Cref{lem:Schreier}, $\Gamma$ is a context-free automaton. Let $\Lambda$ be the contextual automaton of $\Gamma$ and, throughout, let $N$ denote the number of states of $\Lambda$. 

We have the following trivial consequence of the pigeonhole principle.

\begin{proposition}\label{prop:pump}
Let $\pi$ in $\Lambda^*_0$ be a path of length at least $N$. Then there exist paths $\pi_-$, $\pi_+$ and a (non-trivial) loop $\pi_*$ in $\Lambda^*$ such that 
\begin{equation}\label{eq:pump}
\pi = \pi_- \, \pi_* \, \pi_+
\end{equation}
with $\pi_- \, \pi_*$ of length at most $N$. 
\end{proposition}

\begin{definition}[Pumping reductions]
Let $(\pi, q)$ be a state of $\Gamma$. We call $(\pi, q)$ \emph{reducible} if $\pi$ can be written in the form (\ref{eq:pump}) (where we allow $\pi$ to have length less than $N$), and otherwise \emph{irreducible}. For a reducible state $(\pi, q)$, the \emph{pumping reduction} of $(\pi, q)$ is $(\pi_-\, \pi_+, q)$, where $\pi_- \, \pi_*$ is chosen to have minimal length. We denote this by 
$$(\pi_- \, \pi_+, q) \impliedby (\pi, q).$$ 
Observe that this reduction is unique for a given $(\pi, q)$, since this is equivalent to $\iota(\pi_*) = \tau(\pi_*)$ being the first state to be visited twice by $\pi$. In particular, we may consider $\pi_-, \pi_*$ and $\pi_+$ as functions of $\pi$, where defined. The \emph{full pumping reduction} of $(\pi, q)$ is the unique sequence
\begin{equation}\label{eq:pumping}
\star \impliedby q \impliedby (\pi_0, q) \impliedby (\pi_1, q) \impliedby \hdots \impliedby (\pi_n, q),
\end{equation}
such that 
$$(\pi_i, q) \impliedby (\pi_{i + 1}, q)$$
is a pumping reduction for every $i = 1, \hdots, n - 1$, $(\pi_0, q)$ is irreducible, and $\pi_n = \pi$. 

For a path $\pi = e_1 \hdots e_n$, define the \emph{truncation} $\pi|_N$ to be the unique prefix of $\pi$ of length $N$ if $n \geq N$ and $\pi$ otherwise. 
\end{definition}

We make the following observation. 

\begin{proposition}\label{prop:uniquepumping}
Let $\pi_0$ and $\pi_1$ be paths in $\Lambda^*_0$, and suppose
$$(\pi_0, q) \impliedby (\pi_1, q)$$
is a pumping reduction. Suppose $(\pi_0', q)$ is a state of $\Gamma$ with $\pi'_0|_N = \pi_0|_N$. Then there exists a unique path $\pi'_1$ in $\Lambda^*_0$ such that $\pi'_1|_N = \pi_1|_N$ and 
$$(\pi'_0, q) \impliedby (\pi'_1, q)$$
is a pumping reduction.
\end{proposition}

\begin{proof}
Observe that we may write $\pi'_0$ uniquely in the form $\pi'_0 = (\pi_0)_- \, \xi$. Then it is clear that $\pi'_1 = (\pi_0)_- \, (\pi_0)_* \, \xi$ is unique such that $(\pi'_1, q)$ satisfies the desired properties.
\end{proof}

\begin{definition}[Pumping automata]
T`he \emph{pumping automaton} of $\Gamma$ is the finite unlabelled automaton $\Delta$ which has states of the forms:
\begin{enumerate}[label = (\roman*)]
\item $\star$, which is also the unique initial state;
\item $q \in Q^t$, for some state $t$ of $\Lambda$;
\item a state of the form $(\pi|_N, q)$, for $(\pi, q)$ a state of $\Gamma$;
\end{enumerate}
and transitions of the forms:
\begin{enumerate}[label = (\roman*)]
\item from $\star$ to $q$, for $q \in Q^t$;
\item from $q \in Q^t$ to a state of the form $(\pi|_N, q)$, where $(\pi, q)$ is irreducible;
\item from $\pi_0|_N$ to $\pi_1|_N$, where $(\pi_0, q) \impliedby (\pi_1, q)$ is a pumping reduction. 
\end{enumerate}
By applying \Cref{prop:uniquepumping} repeatedly, a full pumping reduction of the form (\ref{eq:pumping}) corresponds uniquely to a path of the form
\begin{equation}\label{eq:newreduction}
\star \longrightarrow q \longrightarrow \pi_0|_N \longrightarrow \pi_1|_N \longrightarrow \hdots \longrightarrow \pi_n|_N
\end{equation}
in $\Delta^*_0$. In particular, we may uniquely identify a state of $\Gamma$ with the path of the form (\ref{eq:newreduction}) in $\Delta^*_0$, and allow $G$ to act upon $\Delta^*_0$ on the right, fixing the paths of lengths 0 and 1 and otherwise acting on $\Delta^*_0$ by identification with the states of $\Gamma$.
\end{definition}

\begin{proposition}\label{prop:goodpump}
Let $s \in \Sigma$, and let $(\pi \, e, q)$ and $(\pi \, \xi, q')$ be states of $\Gamma$. Suppose that $(\pi \, e, q)$ and $(\pi \, \xi, q')$ have full pumping reductions
$$\star \impliedby q \impliedby (\pi_0, q) \impliedby (\pi_1, q) \impliedby \hdots \impliedby (\pi_n, q)$$
and
$$\star \impliedby q' \impliedby (\pi'_0, q') \impliedby (\pi'_1, q') \impliedby \hdots \impliedby (\pi'_m, q')$$
respectively. Suppose that there exists a maximal $k$ such that $\pi_{n - k}$ has length at least $N + 1$. Then $n - k \leq N + 1$, and there exist paths $\pi''_i$ in $\Delta^*_0$ of length at least $N$ for $i = 0, \hdots, k$ such that 
$$\pi_{n - i} = \pi''_i \, e, \qquad \pi'_{m - i} = \pi''_i \, \xi.$$ 
In particular, for $i = 0, \hdots, k$, $\pi_{n - i}|_N = \pi'_{m - i}|_N$. 
\end{proposition}

\begin{proof}
The fact that $n - k \leq N + 1$ is immediate from the fact that the lengths of $\pi_i$ are monotone increasing in $i$. As for the second claim, we proceed by induction upon $i$. Suppose that we have the claim for $i$, and that $\pi_{n - (i + 1)}$ has length at least $N + 1$. By the induction hypothesis, $\pi_i''$ has length at least $N$, and so there exists a unique pumping reduction
$$(\pi''_{i + 1}, q) \impliedby (\pi''_i, q).$$
Thus both
$$(\pi''_{i + 1} \, e) \impliedby (\pi''_i \, e, q), \qquad (\pi''_{i + 1} \, \xi, q') \impliedby (\pi''_i \, \xi, q')$$
are pumping reductions, whence as $\pi'_{n - i} = \pi''_i \, e$ and $\pi'_{m - i} = \pi''_i \, \xi$, we deduce that $\pi'_{n - (i + 1)} = \pi''_{i + 1} \, e$ and $\pi'_{m - (i + 1)} = \pi''_{i + 1} \, \xi$ by the uniqueness of pumping reductions. In particular, as $\pi'_{n - (i + 1)} = \pi''_{i + 1} \, e$ has length at least $N + 1$, $\pi''_{i + 1}$ has length at least $N$, and the claim follows by induction. In particular,
$$\pi_{n - i}|_N = \pi'_{m - i}|_N = \pi''_i|_N$$
and the result follows.
\end{proof}

\begin{definition}[Short prefixes]
Let $\delta$ be a path in $\Delta^*_0$ which corresponds to the full pumping reduction
$$\star \impliedby q \impliedby (\pi_0, q) \impliedby (\pi_1, q) \impliedby \hdots \impliedby (\pi_n, q).$$
Suppose $\pi_n$ has length at least $N + 1$, and let $j$ minimal such that $\pi_j$ has length at least $N + 1$. We call the path
$$\star \longrightarrow q \longrightarrow \pi_0|_N \longrightarrow \pi_1|_N \longrightarrow \hdots \longrightarrow \pi_j|_N$$
in $\Delta^*_0$ the \emph{short prefix} of $\delta$, which is unique where it exists. Observe that we may similarly define the short prefix of an element of $\Delta^\omega_0$, and that short prefixes have length at most $N + 3$ (since by \Cref{prop:goodpump}, $j \leq N + 1$). Moreover, every path of length at least $N + 3$ (and in particular, every right-infinite path) of $\Delta^{\leq \omega}_0$ has a short prefix. 
\end{definition}

\begin{proposition}\label{prop:finalhurdle}
Let $\delta'$ in $\Delta^*_0$ be a short prefix. Then for every path $\zeta$ in $\Delta^*$ with $\tau(\delta') = \tau(\delta'') = \iota(\zeta')$, then $(\delta' \, \zeta) \cdot \rho(s) = (\delta' \cdot \rho(s)) \, \zeta$. 
\end{proposition}

\begin{proof}
Let $\delta = \delta'\, \zeta$. Observe that $\delta$ has length at least $3$, and so in particular corresponds to a state of the form $(\pi \, e, q)$ in $\Gamma$. By \Cref{lem:confreemain}, there exists a path $\xi$ in $\Lambda^*$ of length at most $2$ and a $q'$ such that $(\pi \, e, q) \cdot \rho(s) = (\pi \, \xi, q')$. Suppose that $(\pi \, e, q)$ and $(\pi \, \xi, q')$ have full pumping reductions
$$\star \impliedby q \impliedby (\pi_0, q) \impliedby (\pi_1, q) \impliedby \hdots \impliedby (\pi_n, q)$$
and
$$\star \impliedby q' \impliedby (\pi'_0, q') \impliedby (\pi'_1, q') \impliedby \hdots \impliedby (\pi'_m, q')$$
respectively. Then the path $\delta$ is
$$\star \longrightarrow q \longrightarrow \pi_0|_N \longrightarrow \pi_1|_N \longrightarrow \hdots \longrightarrow \pi_n|_N,$$
and the path $\delta \cdot \rho(s)$ is
$$\star \longrightarrow q' \longrightarrow \pi'_0|_N \longrightarrow \pi'_1|_N \longrightarrow \hdots \longrightarrow \pi'_m|_N.$$
Let $k$ be maximal such that $\pi_{n - k}$ has length at least $N + 1$, namely, so that $\delta'$, the short prefix of $\delta$, is precisely
$$\star \longrightarrow q \longrightarrow \pi_0|_N \longrightarrow \pi_1|_N \longrightarrow \hdots \longrightarrow \pi_{n - k}|_N.$$
By \Cref{prop:goodpump}, the paths
$$\pi_{n - k} |_N \longrightarrow \pi_{n - k + 1}|_N \longrightarrow \hdots \longrightarrow \pi_n|_N$$
and 
$$\pi'_{m - k}|_N \longrightarrow \pi'_{m - k + 1}|_N \longrightarrow \hdots \longrightarrow \pi'_m|_N$$
are both equal, in particular, equal to $\zeta$. Let $\delta''$ denote the path
$$\star \longrightarrow q' \longrightarrow \pi'_0|_N \longrightarrow \pi'_1|_N \longrightarrow \hdots \longrightarrow \pi'_{m - k}|_N.$$
In this case, we have that $(\delta' \, \zeta) \cdot \rho(s) = \delta'' \, \zeta$.

We claim that $\delta'' = \delta' \cdot \rho(s)$. Observe that, with $\pi''_i$ as in \Cref{prop:goodpump}, $\pi_{n - k} = \pi''_k \, e$ and $\pi_{m - k} = \pi''_k \, \xi$, so by \Cref{lem:confreemain}, $(\pi_{n - k}, q) \cdot \rho(s) = (\pi'_{m - k}, q')$. In particular, as $(\pi_{n - k}, q)$ and $(\pi'_{m - k}, q')$ are identified with $\delta'$ and $\delta''$ respectively, we have that $\delta'' = \delta' \cdot \rho(s)$, and the result follows.
\end{proof}

\begin{lemma}\label{lem:finalstep}
There exists a homomorphism $\phi: G \rightarrow [\mcal{G}(\Delta_+)]$ with kernel $\core(G_*)$. Moreover, $G_* = \phi^{-1}\,[\mcal{G}(\Delta_+)]_*$ for some stabiliser $[ \mcal{G}(\Delta_+) ]_*$ of $[ \mcal{G}(\Delta_+) ]$ under the action of $(\Delta_+)^\omega_0$, for a point in $\Delta_0^*$.
\end{lemma}

\begin{proof}
We claim that the action of $G$ upon $\Delta^*_0$ extends continuously to a unique action upon $(\Delta_+)^\omega_0$ by elements of $[ \mcal{G}(\Delta_+) ]$, whence we obtain a homomorphism $\phi: G \rightarrow [\mcal{G}(\Delta_+)]$. In this case, we are done, since $G_*$ is the preimage of the stabiliser of the point of $\Delta^*_0$ corresponding to the unique state of $\vertti \Gamma$, under the action of $G$ upon $(\Delta_+)^\omega_0$. In particular, the action of $G$ on $\Delta^*_0$ has kernel $\core(G_*)$ by \Cref{rem:schreier}, and hence as $\Delta^*_0$ is dense in $(\Delta_+)^\omega_0$ by \Cref{prop:density}, $\phi$ has kernel $\core(G_*)$ as well.

Let $G$ act upon $\Delta^\omega_0$, as contained in $(\Delta_+)^\omega_0$, as follows. Let $\delta'$ be a short prefix in $\Delta^*_0$, and for every $\zeta \in \Delta^\omega$ with $\iota(\zeta) = \tau(\delta')$ and $s \in \Sigma$,
$$(\delta' \, \zeta) \cdot \rho(s) = (\delta' \cdot \rho(s)) \, \zeta.$$
This is well-defined by the fact that every right-infinite path has a short prefix. Moreover, by \Cref{prop:finalhurdle} and the density result of \Cref{prop:density}, it is the unique continuous extension of the action of $G$ on $\Delta^*_0$ to $\Delta^\omega_0$, so is in particular is also a group action in and of itself.

It now suffices to show that $G$ acts by elements of $[ \mcal{G}(\Delta_+) ]$. Let $s \in \Sigma$, and let $\delta$ be a path in $(\Delta_+)^\omega_0$. It suffices to show that there exists an open neighbourhood $U$ of $\delta$ such that $\phi(\rho(s))|_U \in \mcal{G}(\Delta_+)$. If $\delta$ belongs to $\Delta^*_0$, by \Cref{prop:density}, the point $\delta$ is isolated, and so we may take $U = \{\delta\}$. Since $\delta \cdot \rho(s)$ also belongs to $\Delta^*_0$ (that is, also has $(e_+)^\omega$ as a suffix), we have that $\rho(s)|_U$ is indeed a prefix map. 

Otherwise, if $\delta$ belongs to $\Delta^\omega_0$, let $\delta'$ be the short prefix of $\delta$, and let $U = \delta' \, (\Delta_+)^\omega$. Then, for $\delta'' = \delta' \cdot \rho(s)$, we have that the restriction of $\phi(\rho(s))$ to $U$ acts by
$$\phi(\rho(s))|_U: \delta' \, (\Delta_+)^\omega \longrightarrow \delta'' \, (\Delta_+)^\omega, \qquad \delta' \, \zeta \longmapsto \delta'' \,\zeta$$
for all $\zeta \in (\Delta_+)^\omega$ with $\iota(\zeta) = \tau(\delta')$, and in particular belongs to $\mcal{G}(\Delta_+)$. The result follows.
\end{proof}

\begin{proof}[Proof of `only if' in \Cref{thm:main}]
Follows immediately from combining \Cref{lem:embed} with \Cref{lem:finalstep}.
\end{proof}

\begin{proof}[Proof of `only if' in \Cref{cor:main}]
Let $G_1, \hdots, G_n$ be context-free subgroups of $G$. Then the maps $\phi_i: G \rightarrow V$ constructed in \Cref{thm:main} have kernels $\core(G_i)$ by \Cref{lem:finalstep}. Therefore, the direct sum of these homomorphisms
$$G \longrightarrow V \oplus V \oplus \hdots \oplus V$$
has kernel 
$$\bigcap_{i = 1}^n \core(G_i) = 1,$$
that is, is an embedding.
By \Cref{rem:vprops}, there exists an embedding
$$V \oplus V \oplus \hdots \oplus V \longhookrightarrow V.$$
Composing these, the result follows.
\end{proof}

\section{The `if' direction}\label{section:pda}

In this section we detail the `if' directions of the proofs of \Cref{thm:main} and \Cref{cor:main}. We adapt our proofs from the papers of \cite{birget} and \cite{lehnert} which, while they do not contain the results we require explicitly, contain results in a similar spirit.

Firstly, we introduce a notion from computer science from \cite[6.1.2]{hopcroft}.

\begin{definition}[Pushdown automata]\label{def:pda}
Let $(Q, \Sigma, Z, \delta, q_0, Q_1, \zeta_0)$ be a $7$-tuple, comprising:
\begin{enumerate}[label = (\roman*)]
\item A finite set $Q$ of \emph{positions};
\item An \emph{input alphabet} $\Sigma$ and \emph{stack alphabet} $Z$;
\item \emph{Transition functions} $\delta_s$ from $(Z \cup \{\ve\}) \times Q$ to (possibly empty) finite subsets of $Z^* \times Q$ for each $s \in \Sigma \cup \{\ve\}$;
\item A choice of an \emph{initial state} $q_0 \in Q$ and \emph{terminal states} $Q_1 \subset Q$ and the \emph{initial stack} $\zeta_0 \in Z^*$. 
\end{enumerate}

The \emph{pushdown automaton} induced by this tuple is defined as follows. Let $\Gamma$ be the automaton whose states are pairs of the form $(\zeta, q)$ for $\zeta \in Z^*$ (the \emph{stack}) and $q \in Q$ (the \emph{position}) and with transitions of the forms:
\begin{enumerate}[label = (\roman*)]
\item from $(\zeta \, z, q)$ to $(\zeta \, \zeta', q')$ labelled by $s$ for all $(\zeta', q') \in \delta_s(z, q)$; and
\item from $(\ve, q)$ to $(\zeta', q')$ labelled by $s$ for all $(\zeta', q') \in \delta_\ve(z, q)$
\end{enumerate}
for all $s \in \Sigma \cup \{\ve\}$, $\zeta \in Z^*$ and $q \in Q$. Take as initial and terminal states 
$$\vertti \Gamma = \{(\zeta_0, q_0)\}, \qquad \verttt \Gamma = \{(\zeta_1, q_1): \zeta_1 \in Z^*, \, q_1 \in Q_1\}.$$
The subautomaton of $\Gamma$ induced upon all accessible states of $\Gamma$ is called a \emph{pushdown automaton}. 
\end{definition}

Note that the automaton does not halt upon reaching an empty stack contrary to some definitions, though this can be seen not to affect the class of languages accepted by automata of this kind.

We have the following classical theorem due to Muller and Schupp in \cite[Theorem 2.6]{ms3}, which justifies the equivalence between the geometric definition of \Cref{def:cfa} and the computing scientific definition of \Cref{def:pda}.

\begin{theorem}\label{thm:ms3}
An automaton $\Gamma$ is isomorphic to a pushdown automaton if and only if it is isomorphic to a context-free automaton.
\end{theorem}

Note that taking $Z = \edge \Lambda$ in \Cref{lem:confreemain} provides the `if' direction of this result.

We begin with the following proposition. 

\begin{proposition}\label{prop:beaut}
Let $\Gamma$ be a pushdown automaton. Then the language
$$\Bigg\{\,
\sigma \in \Sigma^*: 
\begin{aligned}
& \textup{there exists a path } \pi \textup{ such that } \\
& \,\iota(\pi) = \tau(\pi) \in \vertti \Gamma, \ell(\pi) = \sigma
\end{aligned}
\Bigg\} \subset \Sigma^*$$
is a context-free language.
\end{proposition}

\begin{proof}
Modify $\Gamma$ so that $\verttt \Gamma = \vertti \Gamma$. It suffices to show that $\Gamma$ remains a context-free automaton. It is clear, however, that $\Gamma$ still has finitely many contexts up to isomorphism.
\end{proof}

Throughout, for $\xi \in \{0, 1\}^*$, let $\xi \mapsto \overline{\xi}$ denote the involution
$$x_1 \hdots x_n \longmapsto x_n \hdots x_1.$$

\begin{definition}[Trailing zeros]
We say $\xi \in \{0, 1\}^{\leq \omega}$ has \emph{trailing zeros} if $0^n$ is a suffix of $\xi$ for some $n \geq 1$ (including $n = \omega$). Observe that every $\xi \in \{0, 1\}^{\leq\omega}$ can be written uniquely in the form $\xi = \xi|_\emptyset \, 0^n$ for some $n \geq 1$ and $\xi|_\emptyset \in \{0, 1\}^{\leq \omega}$ with no trailing zeros.
\end{definition}

In the following, we will construct an automaton, where, for $\xi \in \{0, 1\}^\omega$ in the orbit of $0^\omega$ under $V$, the stack $\zeta = \overline{\xi|_\emptyset} \in \{0, 1\}^*$ will represent the element $\xi \in \{0, 1\}^\omega$. This rectifies the fact that $V$ acts by \emph{prefix} partition maps upon elements of $\{0, 1\}^\omega$, while pushdown automata act upon the \emph{suffixes} of finite stacks.

Let $\rho: \Sigma^* \rightarrow V$ be a symmetric, surjective monoid homomorphism. For $s \in \Sigma$, let $\mcal{A}_s, \mcal{B}_s \subset \{0, 1\}^*$ be prefix partitions, and $f_s: \mcal{A}_s \rightarrow \mcal{B}_s$ a bijection such that $[f_s] = \rho(s)$. Define a pushdown automaton as follows. Take as the input alphabet $\Sigma$, stack alphabet $Z = \{0, 1\}$ and define 
$$Q = \{q_*\} \cup \{(s, \alpha): \alpha \textrm{ is a prefix of an element of } \mcal{A}_s\}.$$
Let $q_0 = q_1 = q_*$ and let $\zeta_0 = \ve$.

We now define the transitions in the following fashion: let
\begin{enumerate}[label = (\roman*)]
\item $\delta_s(z, q_*) = \{(z, (s, \ve))\}$;
\item $\delta_\ve(z, (s, \alpha)) = \{(\ve, (s, \alpha \, z))\}$ for $\alpha \notin \mcal{A}_s$; 
\item $\delta_\ve(\ve, (s, \alpha)) = \{(\ve, (s, \alpha \, 0)\}$ for $\alpha \notin \mcal{A}_s$; 
\item $\delta_\ve(z, (s, \alpha)) = \{(z \, \overline{f_s(\alpha)}, q_*)\}$ for $\alpha \in \mcal{A}_s$;
\item $\delta_\ve(\ve, (s, \alpha)) = \{(\overline{f_s(\alpha)|_\emptyset}, q_*)\}$ for $\alpha \in \mcal{A}_s$ 
\end{enumerate}
for all $s \in \Sigma$ and $z \in Z$ and otherwise let $\delta$ evaluate to the empty set. Henceforth, let $\Gamma$ be the pushdown automaton induced by the tuple $(Q, \Sigma, Z, \delta, q_0, Q_1, \zeta_0)$.

\begin{proposition}\label{prop:quitecool}
Let $s \in \Sigma$ and let $\xi, \xi' \in \{0, 1\}^\omega$ be in the orbit of $0^\omega$. Let $\zeta = \overline{\xi|_\emptyset}$ and $\zeta' = \overline{\xi'|_\emptyset}$. Then there exists a path $\pi$ in $\Gamma$ from $(\zeta, q_*)$ to $(\zeta', q_*)$ with $\ell(\pi) = s$ if and only if $\xi \cdot \rho(s) = \xi'$. 
\end{proposition}

\begin{proof}
Suppose $\pi = e_1 \hdots e_n$ is a path in $\Gamma$ with $\iota(\pi) = (\zeta, q_*)$, $\tau(\pi) = (\zeta', q_*)$ and $\ell(\pi) = s$. Observe that the only transitions with initial state $(\zeta, q_*)$ have label $s$. Therefore, since $\ell(\pi) = \ell(e_1) \hdots \ell(e_n)$, we deduce that $e_1$ is a transition of type (i) with $\ell(e_1) = s$ and $e_i$ is a transition of types (ii)$-$(v) with $\ell(e_i) = \ve$ otherwise (since, if $\pi$ contained more than $1$ transition of type (i), then $\ell(\pi)$ would have length at least 2, which is absurd). Moreover, the path $\pi$ can only visit a state with position $q_*$ initially and terminally (or otherwise $\pi$ would have more than $1$ transition of type (i), which is again absurd).

Suppose that $\xi = \alpha \, \eta$ for $\alpha \in \mcal{A}_s$ and $\eta \in \{0, 1\}^\omega$. Suppose that $\alpha = a_1 \hdots a_m$ for $a_i \in \{0, 1\}$. We claim that
\begin{equation}\label{eq:super}
\tau(e_i) = (\overline{(a_i \hdots a_m \, \eta)|_\emptyset}, (s, a_1 \hdots a_{i - 1}))
\end{equation}
for $i = 1, \hdots, m + 1$. We proceed by induction. The case $i = 1$ is clear, since $e_1$ is a transition of type (i). Now suppose that \eqref{eq:super} holds for some $1 \leq i \leq m$. Observe that $a_1 \hdots a_{i - 1} \neq \mcal{A}_s$ and so $e_{i + 1}$ must be a transition of type (ii) or (iii).

If $e_{i + 1}$ is a transition of type (ii), then $z = a_i$ is the final character of the stack. Since $a_1 \hdots a_{i - 1}$ does not belong to $\mcal{A}_s$, the transition $e_{i + 1}$ removes the final character $a_i$ from the stack and changes the position to $(s, a_1 \hdots a_i)$, as desired.

If $e_{i + 1}$ is a transition of type (iii), then $z = \ve$, since the stack is empty. In particular, $a_i = 0$, since 
$$\overline{(a_i \hdots a_m \, \eta)|_\emptyset} = \ve.$$
Again, since $a_1 \hdots a_{i - 1}$ does not belong to $\mcal{A}_s$, we have that 
\begin{align*}
\tau(e_{i + 1}) &= (\ve, (s, a_1 \hdots a_{i - 1} \, 0))\\
&= (\overline{(a_{i + 1} \hdots a_m \, \eta)|_\emptyset}, (s, a_1 \hdots a_{i - 1} \, a_i))
\end{align*}
as required. We deduce that 
$$\tau(e_{m + 1}) = (\overline{\eta|_\emptyset}, (s, a_1 \hdots a_m)).$$
Therefore, since $a_1 \hdots a_m = \alpha \in \mcal{A}_s$, we deduce that $e_{m + 2}$ must be a transition of type (iv) or (v). In particular, $n = m + 2$ and $\tau(e_{m + 2}) = \tau(\pi) = (\zeta', q_*)$ (since $\pi$ only visits a state of the form $(\zeta'', q_*)$ initially or terminally).

If $e_{m + 2}$ is a transition of type (iv), then $\overline{\eta|_\emptyset} \neq \ve$. Hence $e_{m + 2}$ appends $\overline{f_s(\alpha)}$ to the end of the stack and changes the position to $q_*$. Therefore
\begin{align*}
\tau(e_{m + 2}) &= (\overline{\eta|_\emptyset} \, \overline{f_s(\alpha)}, q_*)\\
&= (\overline{(f_s(\alpha) \, \eta)|_\emptyset}, q_*)\\
&= (\overline{\xi'|_\emptyset}, q_*)
\end{align*}
for $\xi \cdot \rho(s) = \xi \cdot [f_s(\alpha)] = \xi'$. 

Similarly, if $e_{m + 2}$ is a transition of type (v), then $\overline{\eta|_\emptyset} = \ve$. Hence $e_{m + 2}$ appends $\overline{f_s(\alpha)}$ to the end of the stack and changes the position to $q_*$. Therefore
\begin{align*}
\tau(e_{m + 2}) &= (\overline{f_s(\alpha)|_\emptyset}, q_*)\\
&= (\overline{(f_s(\alpha) \, \eta)|_\emptyset}, q_*)\\
&= (\overline{\xi'|_\emptyset}, q_*)
\end{align*}
for $\xi \cdot \rho(s) = \xi \cdot [f_s(\alpha)] = \xi'$. From these two cases, the `only if' follows. This also gives a construction of such a path $\pi$ for $\xi \cdot \rho(s) = \xi'$ and so shows the `if' direction, too.
\end{proof}

\begin{proposition}\label{prop:reallycool}
Let $\sigma \in \Sigma^*$ and let $\xi, \xi' \in \{0, 1\}^\omega$ be in the orbit of $0^\omega$. Let $\zeta = \overline{\xi|_\emptyset}$ and $\zeta' = \overline{\xi'|_\emptyset}$. Then there exists a path $\pi$ in $\Gamma$ with label $\sigma$ from $(\zeta, q_*)$ to $(\zeta', q_*)$ if and only if $\xi \cdot \rho(\sigma) = \xi'$.
\end{proposition}

\begin{proof}
The `if' follows from considering $\sigma = s_1 \hdots s_n$ for $s_i \in \Sigma$ and concatenating the paths constructed in \Cref{prop:quitecool}. Thus it suffices to show the `only if'.

Suppose $\pi$ is a path with $\iota(\pi) = (\zeta, q_*)$, $\tau(\pi) = (\zeta', q_*)$ and $\ell(\pi) = \sigma$. Observe that we may decompose $\pi = \pi_1 \hdots \pi_n$ uniquely such that $\pi$ only encounters a state with position $q_*$ at states of the form $\iota(\pi_i)$ and $\tau(\pi_i)$. 

We claim that $\ell(\pi_i)$ has length $1$ in $\Sigma^*$. To see this, observe that a transition has label $\ve$ unless its initial state has position $q_*$. Therefore, by definition of the $\pi_i$, if $\pi_i = e_1 \hdots e_k$, then $e_1$ must be a transition of type (i), while all subsequent transitions in $\pi_i$ must have types (ii)$-$(v) and hence have $\ell(e_j) = \ve$ for $j > 1$. Therefore, we may suppose $s_i = \ell(\pi_i)$ for $s_i \in \Sigma$. 

Suppose that $\tau(\pi_i) = (\zeta_i, q_*)$ with $\zeta_i = \overline{\xi_i|_\emptyset}$ for $\xi_i \in \{0, 1\}^\omega$ in the orbit of $0^\omega$. Then by \Cref{prop:quitecool}, we deduce that $\xi_1 = \xi \cdot \rho(s_1)$ and $\xi_{i + 1} = \xi_i \cdot \rho(s_{i + 1})$ for $i = 1, \hdots, n - 1$, whence 
$$\xi' = \xi \cdot \rho(s_1) \hdots \rho(s_n) = \xi \cdot \rho(\sigma)$$
and the result follows.
\end{proof}

\begin{proof}[Proof of `if' in \Cref{thm:main}]
By \Cref{prop:pullback}, it suffices to show that $V_*$ is a context-free subgroup of $V$. This is immediate from \Cref{prop:beaut} for the automaton $\Gamma$, since by \Cref{prop:reallycool}, for $\sigma \in \Sigma^*$, there exists a path from $(q_*, \ve)$ to $(q_*, \ve)$ in $\Gamma$ with label $\sigma$ if and only if $0^\omega \cdot \rho(\sigma) = 0^\omega$. Hence $\rho^{-1} \, V_* \subset \Sigma^*$ is a context-free language.
\end{proof}

\begin{proof}[Proof of `if' in \Cref{cor:main}]
From the `if' direction of \Cref{thm:main} and \Cref{lem:cftrsub}, it suffices to show that $\core(V_*)$ is trivial. But by \Cref{rem:vprops}, $\core(V_*)$ is the subgroup of $V$ which stabilises the orbit of $0^\omega$ pointwise. Since this orbit is dense in $\{0, 1\}^\omega$ and $V$ acts continuously upon $\{0, 1\}^\omega$, $\core(V_*)$ is trivial.
\end{proof}

\section{Applications}

We now discuss a brief application of both directions of \Cref{thm:main} in constructing an interesting class of context-free subgroups. Our primary motivation is to show how \Cref{thm:main} allows us to understand dynamics upon $V$ by recourse to context-free subgroups and vice versa. We refer the reader to \cite[14.2.4]{cocf} for the original discussion of the relationship between context-free languages and wreath products.

We have the following.

\begin{proposition}\label{prop:demon}
Let $H$ be a finitely generated group and suppose that $H_*$ is a context-free subgroup of $H$. Then there exists a homomorphism $\phi: H \rightarrow V$ and an $n \geq 1$ such that $\phi$ is demonstrative with respect to $H_*$ for an open set $U = 0^n \, \{0, 1\}^\omega$.
\end{proposition}

\begin{proof}
The existence of a homomorphism $\phi: H \rightarrow V$ which is demonstrative with respect to $H_*$ for some $U \subset \{0, 1\}^\omega$ follows from combining \Cref{lem:embed} and \Cref{lem:finalstep}, observing that the embedding constructed in \Cref{lem:embed} is in fact demonstrative. Moreover, by \Cref{rem:vprops}, there exists an element $\alpha \, 0^\omega \in U$ and some $g \in V$ with $0^\omega \cdot g = \alpha \, 0^\omega$. Therefore, conjugating $\phi$ by $g$, we may suppose that $0^\omega \in U$ and so $0^n \, \{0, 1\}^\omega \subset U$ for some $n \geq 1$. By replacing $U$ by the open subset $0^n \, \{0, 1\}^\omega$, the result follows.
\end{proof}

\begin{definition}[Supports]
Let $g \in V$. The \emph{support} of $g$ is the open set
$$\supp(g) = \{\xi \in \{0, 1\}^\omega: \, \xi \, g \neq \xi\}.$$
Observe that, for $g, h\in V$, if $\supp(g)$ and $\supp(h)$ are disjoint, then $g$ and $h$ commute. Moreover, $\supp(h^{-1} \, g \, h) = \supp(g) \, h$. 
\end{definition}

\begin{definition}[Generalised wreath products]
For groups $G, H$ and $H_*$ a subgroup of $H$, define their \emph{generalised wreath product} by
$$G \wr_{H_*} H = \bigg(\bigoplus_{H_* \setminus H} G \bigg) \rtimes H$$
where, for $\widetilde{h} \in H$ and $g_h \in H$,
$$\widetilde{h}^{-1} \, \bigg(\bigoplus_{H_* \setminus H} g_h\bigg) \, \widetilde{h} = \bigoplus_{H_* \setminus H}g_{h \, \widetilde{h}^{-1}}.$$
Here, the direct sums are quantified over cosets $H_* \, h$ for $h \in H$ distinct coset representatives of $H_* \setminus H$. 
\end{definition}

\begin{lemma}\label{thm:demon}
Let $G, H$ be finitely generated groups and $\phi_0: G \rightarrow V$, $\phi_1: H \rightarrow V$ are homomorphisms such that $\phi_0^{-1}\, V_* = G_*$ and $\phi_1^{-1} \, V_* = H_*$. Suppose further that, for some $n \geq 1$, $\phi_1$ is demonstrative with respect to $H_*$ for an open set $U = 0^n \, \{0, 1\}^\omega$. Let $\theta$ denote the homeomorphism 
$$\theta: \{0, 1\}^\omega \longrightarrow U, \qquad \xi \longmapsto \alpha \, \xi.$$
Then there exists a homomorphism 
$$\phi: G \wr_{H_*} H \longrightarrow V$$
which agrees with $\phi_1$ on $H$ and which is defined on $G$ by 
$$\xi \cdot \phi(g) = 
\begin{cases}
\xi \cdot \theta^{-1} \, \phi_0(g) \, \theta & \textrm{for } \xi \in U; \\
\xi & \textrm{otherwise.}
\end{cases}$$
Moreover, 
$$\phi^{-1}\, V_* = \bigg(\bigg(\bigoplus_{H_* \setminus (H - H_*)} G\bigg) \oplus G_*\bigg) \rtimes H_*.$$
\end{lemma}

\begin{proof}
For $g \in G$, let $\widetilde{\phi}_0(g)$ denote the element of $V$ defined by
$$\xi \cdot \widetilde{\phi}_0(g) = \begin{cases}
\xi \cdot \theta^{-1} \, \phi_0(g) \, \theta & \textrm{for } \xi \in U; \\
\xi & \textrm{otherwise.}
\end{cases}$$
for $\xi \in \{0, 1\}^\omega$. Clearly $\widetilde{\phi}_0$ is a homomorphism and for all $g \in G$, $\supp(\widetilde{\phi}_0(g)) \subset U$. 

Observe that, since every $h_* \in H_*$ has that $\phi_1(h_*)$ fixes $U$ pointwise, for $g \in G$ and $h \in H$,
$$\widetilde{\phi}_h(g) = \phi_1(h)^{-1} \, \widetilde{\phi}_0(g) \, \phi_1(h)$$
does not depend upon the representative of a coset $H_* \, h$. Also observe that, for $g \in G$ and $h \in H$, the supports $\supp(\widetilde{\phi}_h(g)) \subset U \cdot \phi_1(h)$ are disjoint for distinct cosets $H_* \, h$. Therefore, for $g, g' \in G$ and $h, h' \in H$ with $H_* \, h \neq H_* \, h'$, $\widetilde{\phi}_h(g)$ and $\widetilde{\phi}_{h'}(g')$ commute. Thus, we may define a homomorphism
$$\phi: \bigoplus_{H_* \setminus H} G \longrightarrow V, \qquad \bigoplus_{H_* \setminus H} g_h \longmapsto \bigoplus_{H_* \setminus H} \widetilde{\phi}_h(g_h)$$
where $g_h \in G$ for $h \in H$ coset representatives of $H_* \setminus H$, which agrees with $\widetilde{\phi}_0$ on $G$.

We now show that $\phi$ extends to a homomorphism from $G \wr_{H_*} H$ which agrees with $\phi_1$ on $H$. But observe that for $\widetilde{h} \in H$ and $g_h \in G$ where $h \in H$ are coset representatives of $H_* \setminus H$, 
\begin{align*}
\phi_1(\widetilde{h})^{-1} \, \phi\bigg(\bigoplus_{H_* \setminus H} g_h \bigg) \, \phi_1(\widetilde{h}) &= \phi_1(\widetilde{h})^{-1} \, \bigg(\bigoplus_{H_* \setminus H} \widetilde{\phi}_h(g_h)\bigg) \, \phi_1(\widetilde{h}) \\
&= \bigoplus_{H_* \setminus H} \widetilde{\phi}_{h \, \widetilde{h}}(g_h)\\
&= \bigoplus_{H_* \setminus H} \widetilde{\phi}_h(g_{h \, \widetilde{h}^{-1}})\\
&= \phi \bigg(\bigoplus_{H_* \setminus H} g_{h \, \widetilde{h}^{-1}}\bigg)\\
&= \phi \bigg(\widetilde{h}^{-1} \, \bigg(\bigoplus_{H_* \setminus H} g_h \bigg) \, \widetilde{h}\bigg).
\end{align*}
Therefore such a $\phi$ exists.

To compute the stabiliser, let $h_* \in H$ and $g_h \in G$ for $h \in H$ coset representatives of $H_*\setminus H$. Then the element
$$\phi \bigg(\bigg(\bigoplus_{H_* \setminus H} g_h\bigg) \, h_* \bigg)$$
fixes $0^\omega$ if and only if $h_*$ fixes $U$ (that is, if $h_* \in H_*$) and if $\widetilde{\phi}_0(g_1)$ fixes $0^\omega$ (that is, $g_1 \in G_*$), since $\theta$ sends $0^\omega$ to itself. Therefore
$$\phi^{-1}\, V_* = \bigg(\bigg(\bigoplus_{H_* \setminus (H - H_*)} G\bigg) \oplus G_*\bigg) \rtimes H_*$$
as required.
\end{proof}

We now have the following. 

\begin{proof}[Proof of \Cref{thm:wreath}]
By \Cref{prop:demon} (which strengthens the `only if' direction of \Cref{thm:main}), there exists $\phi_1: H \rightarrow V$ such that $\phi_1$ is demonstrative with respect to $H_*$ for an open set $U = 0^n \, \{0, 1\}^\omega$ for some $n \geq 1$. Therefore, by \Cref{thm:demon}, there exists a map $\phi: G \wr_{H_*} H \rightarrow V$ such that 
$$\phi^{-1} \, V_* = \bigg(\bigg(\bigoplus_{H_* \setminus (H - H_*)} G\bigg) \oplus G_*\bigg) \rtimes H_*,$$
which is a context-free subgroup of $G \wr_{H_*} H$ by the `if' direction of \Cref{thm:main}.
\end{proof}

\bibliography{ref}

@article{bleakcocf,
	title        = {Embeddings into {Thompson's} group {$V$} and \textit{coCF} groups},
	author       = {Bleak, C. and Matucci, F. and Neunhöffer, M.},
	year         = 2016,
	journal      = {J. Lond. Math. Soc.},
	volume       = 94,
	number       = 2,
	pages        = {583--597}
}

@article{cfp,
	title        = {Introductory notes on {Richard Thompson}'s groups},
	author       = {Cannon, J. W. and Floyd, W. J. and Parry, W. R.},
	year         = 1996,
	journal      = {Enseign. Math.},
	volume       = 42,
	pages        = {215--256},
	issue        = 2
}

@article{demonstrative,
	title        = {A dynamical definition of f.g. virtually free groups},
	author       = {Bennett, D. and Bleak, C.},
	year         = 2016,
	journal      = {Internat. J. Algebra Comput.},
	volume       = 26,
	number       = 1,
	pages        = {105--121}
}

@article{hardest,
	title        = {The Hardest Context-Free Language},
	author       = {Greibach, S. A.},
	year         = 1973,
	journal      = {SIAM J. Comput.},
	volume       = 2,
	number       = 4
}

@book{hopcroft,
	title        = {Automata Theory, Languages, and Computation},
	author       = {Hopcroft, J. E. and Motwani, R. and Ullman, J. D.},
	year         = 2007,
	publisher    = {Addison-Wesley},
	edition      = 3
}

@article{lehnert,
	title        = {The co-word problem for the {Higman-Thompson} group is context-free},
	author       = {Lehnert, J. and Schweitzer, P.},
	year         = 2007,
	journal      = {Bull. Lond. Math. Soc.},
	volume       = 39,
	number       = 2,
	pages        = {235--241}
}

@phdthesis{lehnertthesis,
	title        = {Gruppen von quasi-Automorphismen},
	author       = {Lehnert, J.},
	year         = 2008,
	school       = {Universit{\"a}tsbibliothek Johann Christian Senckenberg}
}

@article{ms1,
	title        = {Pushdown Automata, Graphs, Ends, Second-Order Logic, and Reachability Problems},
	author       = {Muller, D. E. and Schupp, P. E.},
	year         = 1981,
	publisher    = {ACM},
	series       = {STOC '81},
	pages        = {46--54}
}

@article{ms2,
	title        = {Groups, the Theory of Ends, and Context-Free Languages},
	author       = {Muller, D. E. and Schupp, P. E.},
	year         = 1983,
	journal      = {J. Comput. System Sci.},
	pages        = {295--310},
	issue        = 26
}

@article{ms3,
	title        = {The Theory of Ends, Pushdown Automata, and Second-Order Logic},
	author       = {Muller, D. E. and Schupp, P. E.},
	year         = 1985,
	journal      = {Theoret. Comput. Sci.},
	pages        = {51--75},
	issue        = 37
}

@article{silberstein,
	title        = {Context-free pairs of groups {I}: Context-free pairs and graphs},
	author       = {Ceccherini-Silberstein, T. and Woess, W.},
	year         = 2012,
	journal      = {European J. Combin.},
	volume       = 33,
	number       = 7,
	pages        = {1449--1466}
}

@book{trees,
	title        = {Trees},
	author       = {Serre, J.-P.},
	year         = 1980,
	publisher    = {Springer-Verlag}
}

@book{wpig,
	title        = {Word Processing in Groups},
	author       = {Epstein, D. B. A. and Cannon, J. W. and Holt, D. F. and Levy, S. V. F. and Paterson, M. S. and Thurston, W. P.},
	year         = 1992,
	publisher    = {Jones and Bartlett}
}

@book{thurston,
	title        = {Three-Dimensional Geometry and Topology},
	author       = {Thurston, W. P.},
	year         = 1997,
	volume       = 1
}

@book{cocf,
	title        = {Groups, Languages and Automata},
	author       = {Holt, D. F. and Rees, S. and Röver, C. E.},
	year         = 2017,
	publisher    = {Cambridge University Press},
	series       = {London Mathematical Society Student Texts}
}

@book{sakar,
	title        = {Elements of Automata Theory},
	author       = {Sakarovitch, J.},
	year         = 2009,
	publisher    = {Cambridge University Press}
}

@misc{birget,
	title        = {On the complexity of the word problem of the {R. Thompson} group {$V$}},
	author       = {Birget, J.},
	year         = 2022,
	howpublished = {arXiv:2203.08592}
}

@misc{bleakagain,
	title        = {Action graphs, semiconjugacy, and non-embedding in {Thompson}'s group {$V$}},
	author       = {Hyde, J. and Skipper, R. and Zaremsky, M. C. B.},
	year         = 2026,
	howpublished = {arXiv:2605.20564v2}
}

@article{cftr,
	author 	     = {D'Angeli, D. and Matucci, F. and Perego, D. and Rodaro, E.},
	title        = {Context-free graphs and their transition groups},
	journal      = {Trans. Lond. Math. Soc.},
	volume       = {13},
	number       = {1},
	pages        = {e70034},
	year         = {2026}
}

@misc{independent,
	title        = {A graph-theoretical characterisation of subgroups of {Thompson's} group {$V$}},
	author       = {Bodart, C. and D'Angeli, D. and Perego, D. and Rodaro, E.},
	year         = 2026,
	howpublished = {arxiv:2608.02111}
}
\bibliographystyle{plain}

\end{document}